\documentclass[10pt]{article}      

\usepackage[dvips]{graphics}      
     
\usepackage{amssymb}      
\usepackage{amsmath}      
\usepackage{amsthm}      
\usepackage{amsbsy}      
\usepackage[dvips]{graphics}      
\usepackage{amscd}      
 
\theoremstyle{plain}      
\newtheorem{theorem}{Theorem}

\newtheorem{lemma}{Lemma}[section]     
\newtheorem{corollary}{Corollary}     
\newtheorem{proposition}{Proposition}     
\newtheorem{conjecture}{Conjecture}      

\newtheorem{definition}{Definition}    
\theoremstyle{remark}      
\newtheorem{remark}{Remark}

\newcommand{\Z}{{\mathbb{Z}}}      
      
\newcommand{\R}{{\mathbb{R}}}      

\newcommand{\vol}{{\rm vol}}

\begin{document}

\title{Asymptotics of generalized Hadwiger numbers}
 \author{      
\begin{tabular}{cc}      
Valentin Boju &  Louis Funar \\      
\small\em Montreal Tech, Institut de Technologie de Montreal & \small \em Institut Fourier BP 74, UMR 5582 \\      
\small\em POBox 78574, Station Wilderton & \small \em University of Grenoble I \\      
\small \em Montreal, Quebec H3S 2W9,  Canada &  \small \em 38402 Saint-Martin-d'H\`eres cedex, France \\      
\small\em e-mail: {\tt valentinboju@montrealtech.org} & \small \em e-mail: {\tt funar@fourier.ujf-grenoble.fr} \\      
\end{tabular}      
}

\maketitle

\begin{abstract}
We give asymptotic estimates for the number of 
non-overlapping homothetic copies of  some centrally symmetric oval $B$ 
which have a common point with a 2-dimensional domain $F$ having rectifiable boundary, 
extending previous work of the L.Fejes-Toth, K.Borockzy Jr., D.G.Larman, 
S.Sezgin, C.Zong and the authors.  The asymptotics compute the length 
of the boundary $\partial F$ in the Minkowski metric determined by $B$. 
The core of the proof consists of  a method for sliding  convex beads along 
curves with positive reach in the Minkowski plane. 
We also prove that level sets are rectifiable subsets, extending a theorem of  
Erd\"os, Oleksiv and Pesin for the Euclidean space to the Minkowski space. 
 
\vspace{0.1cm}
\noindent MSC (AMS) Subject Classification: 52 C 15, 52 A 38, 28 A 75.  
\end{abstract}

\section{Introduction}

For closed topological disks $F, B \subseteq \R^{d}$, we denote by 
$N_{\lambda}(F,B)\in \Z_+$ the following generalized Hadwiger number. 
Let $A_{F,B, \lambda}$ denote the family of all sets, homothetic to
$B$ in the ratio $\lambda$, which have only boundary points in common with 
$F$. Then $N_{\lambda}(F,B)$ is the greatest
integer $k$ such that $A_{F, B, \lambda}$ contains $k$ sets with 
pairwise disjoint interiors. In particular, $N_{1}(F,F)$
is the Hadwiger number of $F$ and $N_{\lambda}(F,F)$ the generalized 
Hadwiger number considered first by Fejes Toth  for polytopes 
in (\cite{Fe1,Fe2}) and 
further in \cite{BF}. Extensive bibliography and results 
concerning this topic can be found in \cite{B}. 
The main concern of this note is to find asymptotic estimates 
for $N_{\lambda}(F,B)$ as $\lambda$ approaches $0$, in terms 
of geometric invariants of $F$ and $B$, as it was done for $F=B$ 
in \cite{BF}, and to seek for the higher order terms.  

\vspace{0.2cm}
\noindent 
Roughly speaking, counting the number of homothetic 
copies of $B$ packed along the surface of a 
$d$-dimensional body $F$ amounts to compute the 
$(d-1)$-area of its boundary, up to a certain density 
factor depending only on $B$. The density factor is especially 
simple when dimension $d=2$. 

\vspace{0.2cm}
\noindent 
 A bounded convex centrally symmetric domain $B$ determines a 
Banach structure on $\R^n$ and thus a metric, usually called 
the Minkowski metric associated to $B$ (see \cite{MSW1,T}).  In particular 
it makes sense to consider the length of curves with respect to the 
Minkowski metric.   

\vspace{0.2cm}
\noindent The main result of this paper states the convergence 
of the number of homothetic copies times the homothety factor
to half of the Minkowski length of $\partial F$, 
in the case of planar domains $F$ having rectifiable boundary. 
In order to achieve this we need  first a regularity result 
concerning level sets that we are able to prove in full generality 
in the first section. This is a generalization of a theorem due 
to Erd\"os, Oleksiv and Pesin for the Euclidean space to the Minkowski 
space. The core of the paper is the second section which 
is devoted to the proof of the main result stated above. 
We first prove it for curves of positive reach (following Federer \cite{Fed}) 
and then deduce the general case from this. 
The remaining sections contain partial results concerning the 
higher order terms for special cases (convex  and positive reach
domains) and an extension of the main result in higher dimensions 
for domains with convex and smooth boundary.

\section{Level sets}
Through out this section $B$  will denote a centrally symmetric  compact 
convex domain in $\R^n$.  
Any such $B$ determines a norm 
$\|\, \|_{B}$ by 
$\| x - y\|_{B} = \|x - y\|/\|o - z\|$, 
where $\| \|$ is the Euclidean norm, $o$ is
the center of $B$ and $z$ is a point on the boundary 
$\partial B$ of $B$ such that  the half-lines $|o z$ and $|x y$ are parallel. 
When equipped with this norm, 
$\R^{n}$ becomes a Banach space whose unit disk is 
isometric to $B$. 
We also denote by $d_B$ the distance 
in the $\|\,\|_B$ norm, called also the Minkowski metric structure on 
$\R^n$ associated to $B$. 
We set $xy$, respectively $|xy$, $|xy|$ for the line, 
respectively half-line and segment determined by the points $x$ and $y$. 
As it is well-known in Minkowski geometry segments are geodesics 
but  when $B$ is not strictly convex one might have also  
other geodesic segments than the usual segments.

\vspace{0.2cm}
\noindent 
The goal of this section is to generalize the Erd\"os theorem 
about the Lipschitz regularity of level sets from the Euclidean space 
to an arbitrary Minkowski space (see \cite{E}). 
We will make use of it only for $n=2$ in the next section but we 
think that the general result is also of independent interest 
(see also \cite{Fu,GP}).  

\vspace{0.2cm}
\noindent 
The theorem for the Euclidean space was stated and the 
beautiful  ideas of the proof were 
sketched by Erd\"os in \cite{E}; forty years later 
the full details were worked out by Oleksiv and Pesin in \cite{OP}.

\begin{theorem}\label{lip}
If the set $M$ is bounded  and $r$ is large enough then the level set 
$M_r=\{x\in \R^n; d_B(x,M)=r\}$ is a Lipschitz hypersurface in the 
Minkowski space. Furthermore, for arbitrary $r>0$ the level set 
$M_r$ is the union of finitely many 
Lipschitz hypersurfaces and in particular it is a $(n-1)$-rectifiable 
subset of $\R^n$. 
\end{theorem}
\begin{proof}
Our proof extends the one given by Erd\"os \cite{E} and  
Oleksiv and Pesin in \cite{OP}. 
Let $r_0$ such that $M\subset B(c,r_0)$, where $B(c,r_0)$ denotes the 
metric ball of radius $r_0$ centered at $c$. 
Consider  first $r$ large enough in terms of $r_0$.

\begin{lemma}\label{ang1}
Let $B(x,r)$ be such that $B(x,r)\cap B(c,r_0)\neq\emptyset$ and 
$B(x,r)\setminus B(c,r_0)\neq\emptyset$. 
Set $\gamma$ for the angle under which we can see $B(c,r_0)$ from $x$. 
Then, for any $\varepsilon >0$ there exists some $r_1(\varepsilon,r_0)$ 
which depends only $B, r_0$ and $\varepsilon$ such that, for 
any $r \geq r_1(\varepsilon,r_0)$ we have $\gamma <\varepsilon$. 
\end{lemma}
\begin{proof}
If $r_{\rm max}$ (respectively $r_{\rm min}$) denotes 
the maximum (respectively minimum) 
Euclidean radius of $B$, then 
\begin{equation}  \sin \frac{\gamma}{2} \leq \frac{r_0r_{\rm max}}{r_1r_{\rm min}}\end{equation}
\end{proof}

\begin{lemma}\label{ang2}
There exists some $\alpha(B)< \pi$ such that 
$\widehat{yxz}\leq \alpha$, for any 
$z\in B(y,r)\setminus {\rm int}B(x,r)$. 
\end{lemma}
\begin{proof}
The problem is essentially two-dimensional as we can cut the two metric balls 
by a 2-plane containing  the line $xy$ and the point $z$. 
Suppose henceforth $B$ is planar and consider 
support lines $l^+$ and $l^-$ parallel to $xy$. 

\vspace{0.2cm}
\noindent
Since $l^+\cap \partial B(x,r)$ is convex it is a segment $|v_1^+v_2^+|$, 
possibly degenerate to one point.  We choose $v_1^+$ to be the farthest 
from $l^+\cap B(y,r)$ among $v_1^+$ and $v_2^+$. 
By symmetry $l^-\cap \partial B(x,r)$ is a parallel segment $|v_1^-v_2^-|$, 
with $v_1^+v_1^-$ parallel to $v_2^+v_2^-$.
Let $v^+$ (and $v^-$) be the midpoint of  
$|v_1^+v_2^+|$ (respectively of $|v_1^-v_2^-|$).
Observe that $x\in |v^+v^-|$. We assume that $v_2^+$ and $v_2^-$ lie in the 
half-plane determined by $v^+v^-$ and containing 
$y$.  

\vspace{0.2cm}
\noindent
We claim then that $\widehat{yxz}\leq \max(\widehat{yxv_1^+}, 
\widehat{yxv_1^-})
$.  This amounts to prove that any $z\in B(y,r)\setminus B(x,r)$ 
should lie in the half-plane determined by  the line $v_1^+v_1^-$ 
and containing $y$, as in the picture below. 

\vspace{0.2cm}
\begin{center}
\includegraphics{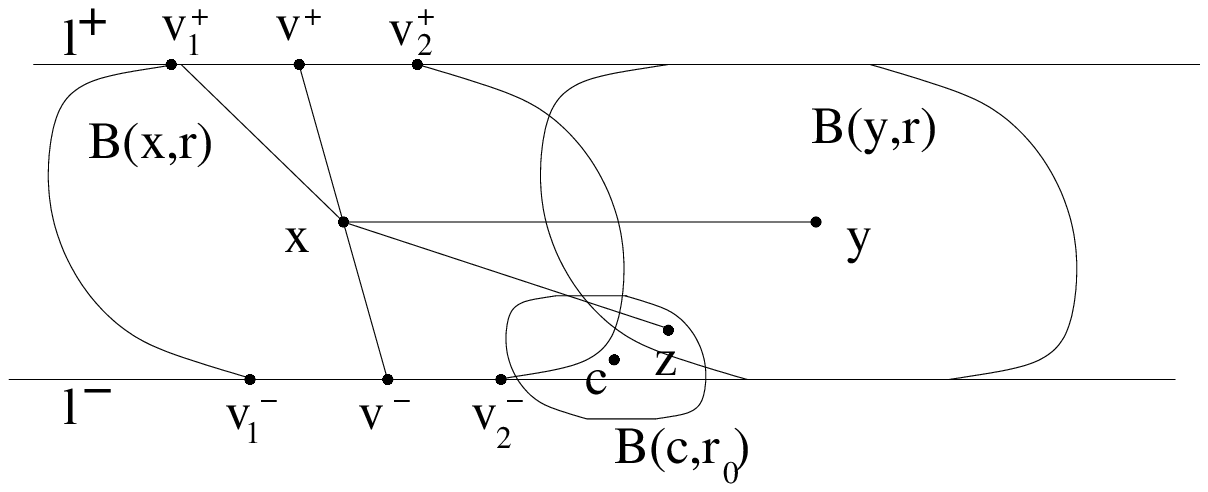}
\end{center}
\vspace{0.2cm}

\vspace{0.2cm}
\noindent
Suppose the contrary, namely that there exists 
$z\in B(y,r)\setminus B(x,r)$ in the  opposite half-plane. 
Let $T$ be the translation in the direction $|yx$ of length $|yx|$. 
We have $T(B(y,r))=B(x,r)$ and $z\in B(y,r)$, hence 
$T(z)\in B(x,r)$.  
The half-line $|T(z)z$ intersects the segment $|v^+v^-|$ in a point 
$w\in B(x,r)$.

\vspace{0.2cm}
\noindent
Suppose first that $B$ is strictly convex.  
Then both $w$ and $T(z)$ belong to $B(x,r)$ while the point 
$z\not\in {\rm int}B(x,r)$. This contradicts the strict  convexity of 
$B(x,r)$, since $T(z)\neq w$. 

\vspace{0.2cm}
\noindent
The direction $v^+v^-$ is called the dual $d^*$ 
of $d=xy$ with respect to $B$ (also called the $B$-orthogonal, 
as introduced by Birkhoff). 

\vspace{0.2cm}
\noindent
It suffices now to remark that for given $B$ the quantity 
$\sup_{\pi} \sup_{d} \max (\angle (d, d^*), \angle(d,-d^*))$, 
the supremum being taken over all planes $\pi$ and all directions $d$, 
is  bounded from above by some $\alpha < \pi$.  In fact the space of 
parameters is a compact (a Grassmannian product the sphere) and that 
this angle cannot be $\pi$ unless the  planar slice degenerates.
 
\vspace{0.2cm}
\noindent
Let us assume now that $B$ is not strictly convex. 
Then the argument above shows that $w$, $T(z)$ belong to $B(x,r)$ 
while the point $z\not\in {\rm int}B(x,r)$. Therefore 
$z\in \partial B(x,r)$ and hence $w,z, T(z)\in \partial B(x,r)$.
Thus $w$ belongs to one of the two support lines $l^+$ or $l^-$.  
By symmetry it suffices to consider the case when  $w=v^+$. 
Since $T(\partial B(y,r)\cap l^+)\subset \partial B(x,r)\cap l^+$ 
it follows that  $\partial B(y,r)\cap l^+$ is the  segment 
$|T^{-1}(v_1^+) T^{-1}(v_2^+)|$. Thus $z$ belongs to the half-plane 
determined by $T^{-1}(v_1^+)$ and $T^{-1}(v_1^-)$, which is 
contained into the one determined by $v_1^+v_1^-$ and containing $y$.

\vspace{0.2cm}
\noindent
The compactness argument above extends to the non strict convex $B$. 
\end{proof}

\begin{remark}
We have $\widehat{yxz}\leq \max(\widehat{yxv^+}, \widehat{yxv^-})$ if 
$z\in B(y,r)\setminus B(x,r)$.  The proof is similar. 
The upper bound is valid for the 
closure of $B(y,r)\setminus B(x,r)$ as well.  
Therefore it holds also for 
$B(y,r)\setminus {\rm int}(B(x,r))$ provided that $B$ is strictly convex, 
but not in general, see for instance the case when $B$ is a rectangle and 
$xy$ is parallel to one side. 
\end{remark}

\vspace{0.2cm}
\noindent
If $\beta$ is an angle smaller than $\frac{\pi}{2}$ we set 
$K(x,\beta, |cx)$ for the cone with vertex $x$ of total angle $2\beta$, of    
axis $|cx$ and going outward $c$.
 
\begin{lemma}
Let us choose $\varepsilon$ such that $\alpha(B)+\varepsilon <\pi$.  
Then for any point $x\in M_r$, with $r\geq r_1(\varepsilon,r_0)$
we have $K(x,\pi-\alpha(B)-\varepsilon, |cx) \cap M_r=\{x\}$. 
\end{lemma}
\begin{proof}
If $x\in M_r$ then ${\rm int}(B(x,r))\cap M=\emptyset$. 
Moreover, $M\subset B(c,r_0)$ and so 
$M\subset B(c,r_0)\setminus {\rm int}(B(x,r))$. 
Let now $y\in M_r$, $y\neq x$.  
Thus $B(y,r)\cap (B(c,r_0)\setminus {\rm int}(B(x,r))\neq \emptyset$. 
Let then $z$ be a point from this set. 

\vspace{0.2cm}
\noindent
Then $z\in B(y,r)\setminus {\rm int}B(x,r)$ so that by lemma \ref{ang2} 
$\angle(yxz)\leq \alpha(B)$. Further lemma \ref{ang1} shows that 
$|\angle(czx)| \leq \varepsilon$, provided that $r\geq r_1(\varepsilon,r_0)$. 
Thus the angle made between the half-lines $|cx$ and $|xy$ 
is at least $\pi-\alpha-\varepsilon$, as can be seen in the figure.

\vspace{0.2cm}
\begin{center}
\includegraphics{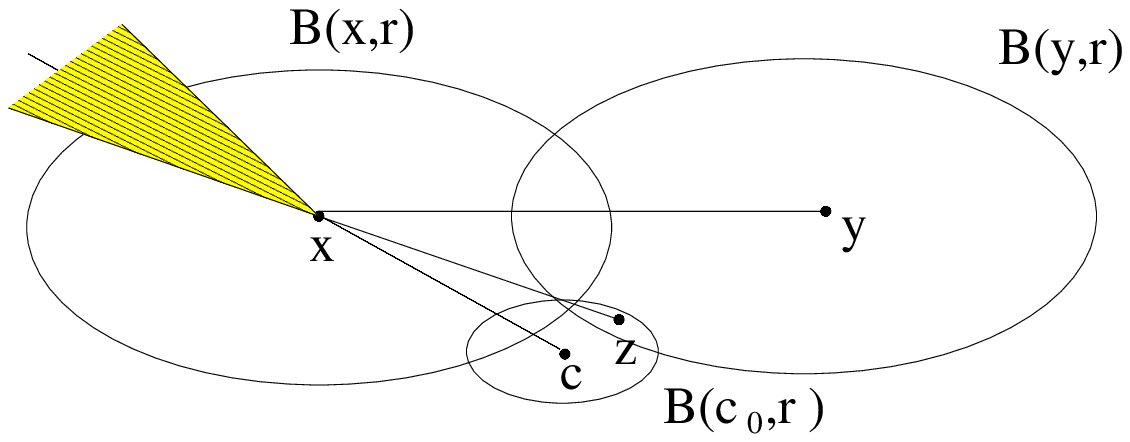}
\end{center}
\vspace{0.2cm}

\vspace{0.2cm}
\noindent
In particular $y$ cannot belong to the  cone 
$K(x,\pi-\alpha(B)-\varepsilon, |cx)$. This proves the lemma. 
\end{proof}

\vspace{0.2cm}
\noindent
{\em Proof of  the theorem}. Set $\beta=
\pi-\alpha(B)-\varepsilon$ and let $r\geq r_1(\varepsilon,r_0)$. 

\vspace{0.2cm}
\noindent
First take any $x\in M_r$ and let $U=M_r\cap K(c, \beta/2, |cx)$. 
If $u\in U$ then $K(u, \beta/2, |cx)\subset K(u, \beta, |cu)$ 
and hence 
\begin{equation} K(u, \beta/2, |cx)\cap M_r \subset K(u, \beta, |cu)\cap M_r =\emptyset\end{equation}
This means that for each $u\in U$ the cone with angle $\beta$ 
and axis parallel to  the fixed half-line $|cx$ contains no other points 
of $U$. Therefore $U$ is the graph of a function of $n-1$ variables 
satisfying a Lipschitz condition with constant equal 
to $\frac{1}{\tan \beta}$. 
 
\vspace{0.2cm}
\noindent
Let consider now the case when $r$ is arbitrary positive.
Choose then $s$ such that $r_1(\varepsilon, s)< r$.  
Split $M$ into a finite number of sets $M_j$ such that each $M_j$ has diameter 
at most $s$. It follows that 
$M_r\subset \cup_{j} {M_j}_r$. 
Since each ${M_j}_r$ is locally Lipschitz it follows that $M$ is 
locally the union of finitely many Lipschitz hypersurfaces.  
\end{proof}

\begin{corollary}\label{lipman}
If $M\subset \R^2$ 
then for almost all $r$ the level set 
$M_r$ is a 1-dimensional  Lipschitz manifold i.e.  the union of  disjoint 
simple closed  Lipschitz curves.   
\end{corollary}
\begin{proof}
In fact Ferry proved (see \cite{Fe}) that for almost all $r$  the level set 
$M_r$ is a 1-manifold. 
\end{proof}
\begin{remark}
Lipschitz curves are precisely those curves which are rectifiable. 
Notice also that the rectifiability does not depend 
on the particular Minkowski metric, as already observed by G\'olab 
(\cite{Go1,Go2}).  
\end{remark}

\begin{remark}
Stach\'o (\cite{St}) proved that level sets $M_r$ in the Minkowski space  
are rectifiable in the sense of Minkowski for all 
but countably many $r$ generalizing earlier results of Sz\"okefalvi-Nagy 
for planar sets.  
\end{remark}

\section{Planar domains: approaching the perimeter}

Unless explicitly stated otherwise, throughout this section, $B$  
will denote a centrally symmetric plane oval, where by oval we mean a 
compact convex domain with non-empty interior.

\vspace{0.2cm}
\noindent 
We assume henceforth that $\partial F$ is  a rectifiable curve, namely 
it is the image of a Lipschitz map from a bounded interval into the plane. 
Set $p_B(\partial F)$ for the length of $\partial
F$ in the norm $\| \,\|_{B}$.

\vspace{0.2cm}
\noindent 
Our main result generalizes theorem 1 from (\cite{BF}), where 
we considered the case $F=B$ and thus $F$ was convex.

\begin{theorem}\label{lim}
For any symmetric oval $B$ and 
topological disk $F$ with rectifiable boundary in the plane, we have
\begin{equation}p_B(\partial F) = 2\lim_{\lambda \to 0}\lambda N_{\lambda}(F,B) 
\end{equation}
\end{theorem}
 
\begin{remark}
The guiding principle of this paper is that we can construct some 
outer packing measure for sets in the Minkowski space which is  
similar to the packing measure defined by Tricot (see \cite{Tricot})  
but uses only {\em equal} homothetic copies of $B$ which are 
packed {\em outside} and hang on the respective 
set. These constraints make it much more rigid than  the measures constructed 
by means of the Caratheodory method (see \cite{Fed}). 
On the other hand it is related to the Minkowski content and 
the associated curvature measures. 

\vspace{0.2cm}
\noindent 
For a fractal set $F$  consider  those $s$ for which 
$\lim_{\lambda\to 0} (2\lambda)^s N_{\lambda}(F,B)$ is finite non-zero. 
If this set consists in a singleton, then call it the Hadwiger dimension 
of $F$ and the  above limit the Hadwiger $s$-measure of $F$.
This measure is actually supported on the ``frontier''
$\partial F$ of $F$.  Although it is not, in general, a bona-fide measure 
but only a pre-measure, there is a standard procedure for 
converting it  into a measure. 
Explicit computations for  De Rham curves show that these make sense  
for a large number of fractal curves. One might expect such 
measures be Lipschitz functions on the space of  measurable curves endowed 
with the Hausdorff metric.   
\end{remark}

\subsection{Curves of positive reach}
Federer introduced in \cite{Fed} subsets of positive reach 
in Riemannian manifolds. His definition extends immediately to Finsler 
manifolds and in particular to Minkowski spaces, as follows:  

\begin{definition}
The closed subset $A\subset R^n$ has positive reach if it admits a neighborhood 
$U$ such that for all $p\in U$ there exists a unique point $\pi(p)\in A$  
which is  the closest point of $A$ to $p$ i.e. such that 
$d_B(p,\pi(p))=d_B(p, A)$.  
\end{definition}
 
\vspace{0.2cm}
\noindent 
It is clear that convex sets and sets with boundary of class ${\mathcal C}^2$ 
have positive reach in the Euclidean space. 
A classical theorem of Motzkin characterized convex sets as those 
sets of positive reach in any Minkowski space whose unit disk $B$ is 
strictly convex and smooth (see \cite{Val}, Theorem 7.8, p.94).  
Moreover, 
Bangert characterized completely in \cite{Ban} 
the sets of positive reach in Riemannian manifolds, as the sub-level sets 
of functions $f$, admitting local charts $(U,\varphi_U:U\to \R^n)$ 
and ${\mathcal C}^{\infty}$ functions $h_U$ such that 
$(f+h_U)\circ \varphi_U^{-1}$ are convex functions.
Another characterization was recently obtained by Lytchak (\cite{Ly}), 
as follows. 
Subsets  $A$ of positive reach in Riemann manifolds are those which are 
locally convex with respect to some Lipschitz continuous Riemann 
metric on the manifold, and equivalently those for which the 
inner metric $d^A$ induced on $A$  by the Riemann distance  verifies 
the inequality 
\begin{equation}  d^A(x,y)\leq d(x,y)(1+C d(x,y)^2)\end{equation}
for any $x,y\in A$ with $d(x,y)\leq \rho$, for some constants 
$C, \rho >0$. 
Federer proved in \cite{Fed} that Lipschitz manifolds
of positive reach are ${\mathcal C}^{1,1}$ manifolds. This was further 
showed to hold true more generally for topological manifolds of positive 
reach (see  \cite{Ly}).  

\vspace{0.2cm}
\noindent 
On the other hand the sets of positive reach might depend on the 
specific Minkowski metric on $\R^n$. For instance if $B$ is a square in $\R^2$ 
then any other rectangle $F$  having an edge parallel to one of $B$ 
has not positive reach. In fact a point in a neighborhood of that edge 
has infinitely many closest points.

\begin{remark}
It seems that sets of positive reach are the same for a Riemannian 
metric on $\R^n$ and the Minkowski metric $d_B$ associated to a 
strictly convex  smooth $B$ (see also \cite{Val} for the extension 
of the Motzkin theorem to Minkowski spaces). 
\end{remark}

\vspace{0.2cm}
\noindent 
We will prove now the main theorem for sets of positive reach:  
\begin{proposition}\label{limreach}
If $\partial F$ is  a Lipschitz curve of positive reach  with respect to 
the Minkowski metric $d_B$ then 
$\lim_{\lambda\to 0}2\lambda N_\lambda(F,B)=p_B(\partial F)$.  
\end{proposition}

\begin{proof} We start by reviewing a number of notations and concepts. 
Let $A_{<\varepsilon}$ (respectively $A_{\leq \varepsilon}$ and 
$A_{\varepsilon}$) denote 
the set of points at distance less than (respectively less or equal than, 
or equal to) $\varepsilon$ from $A$, in the metric $d_B$. 

\vspace{0.2cm}
\noindent 
Recall from \cite{Fed} the following definition: 
\begin{definition} 
The reach $r(A)$ of the set $A$ is  defined to be the larger 
$\varepsilon$ (possibly $\infty$) 
such that each point $x$ of the open neighborhood  
$A_{<\varepsilon}$ has a unique $\pi(x)\in A$ 
realizing the distance from $x$ to $A$.  
\end{definition}

\vspace{0.2cm}
\noindent 
Assume from now that $F$ is a planar domain such that  
$\partial F$ is a Lipschitz curve which  has positive reach. 
We will consider henceforth only those values 
of $\lambda>0$ for which  $2\lambda < r(\partial F)$.

\begin{definition}
Elements of  $A_{F,B, \lambda}$ are called {\em beads} (or $\lambda$-beads 
if one wants to specify the value of $\lambda$) and a configuration 
of $\lambda$-beads with disjoint interiors is called a 
$\lambda$-{\em necklace}. 
The necklace is said to be {\em complete} (respectively {\em almost complete}) 
if all (respectively all but one) pairs of consecutive beads 
have a common point. 
A necklace is {\em maximal} if it contains $N_{\lambda}(F,B)$ beads. 
\end{definition}

\vspace{0.2cm}
\noindent 
The main step in proving  the proposition is to establish first: 
\begin{proposition}\label{necklace}
If $\partial F$ is Lipschitz  and has positive reach 
then there exist maximal almost complete 
$\lambda$-necklaces for any $\lambda <\frac{1}{2}r(\partial F)$. 
\end{proposition}

\vspace{0.2cm}
\noindent 
Consider  now a maximal almost complete necklace and 
$P(\lambda)$ be the associated polygon whose vertices are the centers of 
the beads. Let $a$ and $c$ denote the 
pair of consecutive vertices of $P(\lambda)$ realizing the 
maximal distance among consecutive vertices.   
These are the centers of those beads $A$ and $C$  of the necklace 
which might not touch each other. 
The distance between the beads $A$ and $C$ is called 
the {\em gap} of the almost complete necklace.

\begin{proposition}\label{gap}
A maximal almost complete $\lambda$-necklace of  the simple 
closed curve $\partial F$ whose  reach is greater than $2\lambda$  
has gap smaller than $3\lambda$. 
Consequently the perimeter $p_B(P(\lambda))$ of $P(\lambda)$ satisfies 
the  following inequalities: 
\begin{equation}  0\leq p_B(P(\lambda))-2\lambda N_{\lambda}(F,B) <3\lambda \end{equation}
\end{proposition}

\vspace{0.2cm}
\noindent 
Observe that the set $(\partial F)_{\lambda}$ has two components, 
namely the one contained in the interior of $F$ and that exterior to $F$. 
We set $\partial^+F_{\lambda}=(\partial F)_{\lambda}\cap (\R^2\setminus F)$. 
Moreover, it is easy to see that  
$\partial^+F_{\lambda}=\partial (F_{\leq \lambda})$.

\begin{proposition}\label{leng}
Suppose that $F$ is a planar domain whose boundary $\partial F$ is rectifiable
(without assuming that the reach is positive).  
Then for any $\lambda>0$ we have: 
\begin{equation}  p_B(\partial^+ F_{\lambda})\leq p_B(\partial F)+ \lambda p_B(\partial B)\end{equation}
\end{proposition}

\vspace{0.2cm}
\noindent {\em Proof of Proposition \ref{limreach} assuming 
Propositions \ref{necklace}, \ref{gap} and \ref{leng}}. 
Recall now that $P(\lambda)$ is a polygon with 
$N_{\lambda}$ vertices inscribed 
in $\partial^+F_{\lambda}$. Each pair of consecutive vertices of the 
polygon determines an oriented arc of $\partial^+F_{\lambda}$. Furthermore, 
each edge corresponds to a pair of consecutive beads and thus  
the arcs associated to different edges of $P(\lambda)$ do not overlap.  
We will show later also that $\partial^+F_{\lambda}$ is connected. 
These imply that the perimeter of $P(\lambda)$ is bounded from above 
by the length of $\partial^+F_{\lambda}$.  Therefore we have the inequalities:  
\begin{equation} p_B(P(\lambda))\leq p_B(\partial^+F_{\lambda}) \leq p_B(\partial F) + \lambda p_B(\partial B)\end{equation}
Let $\lambda$ goes to $0$. We derive that:  
\begin{equation} \lim_{\lambda\to 0} p_B(P(\lambda)) \leq p_B(\partial F)\end{equation}

\vspace{0.2cm}
\noindent 
On the other hand recall that $P(\lambda)$ converges to $\partial F$ 
since the distance between consecutive vertices is bounded by $2\lambda$. 
Using the fact that the  
Lebesgue-Minkowski length is 
lower semi-continuous (see \cite{Ce}) we find that: 
\begin{equation}  \lim_{\lambda \to 0}\inf p_B(P(\lambda)) \geq p_B(\partial F)\end{equation}
The two inequalities above imply that $\lim_{\lambda\to 0} p_B(P(\lambda))$ 
exists and is equal to $p_B(\partial F)$. In particular
\begin{equation}  \lim_{\lambda\to 0} 2\lambda N_{\lambda}(F,B)=\lim_{\lambda \to 0} p_B(P(\lambda)) = p_B(\partial F)\end{equation}
and Proposition \ref{limreach} is proved. 
\end{proof}

\begin{remark}
One can also consider packings with disjoint homothetic 
copies of $B$ lying in $F$ and having a common point with 
the complement $\R^2-{\rm int}(F)$. Then a similar asymptotic result 
holds true. 
\end{remark}

\subsection{Proof of Proposition \ref{necklace}}
Consider a maximal  necklace and join consecutive centers of beads 
by segments to obtain a polygon $P(\lambda)$. 
We want to slide the beads  along $\partial F$ so that 
all but at one pairs of consecutive beads  
have a common boundary point.  Observe that $P(\lambda)$ is a polygon with 
$N_{\lambda}=N_{\lambda}(F,B)$ vertices inscribed 
in $\partial^+F_{\lambda}$.

\vspace{0.2cm}
\noindent 
Let $\pi:\partial^+F_{\lambda}\to \partial F$ be the map 
that associates to  the point $x$ the closest point $\pi(x)\in \partial F$. 
Since $\lambda< r(\partial F)$ the map $\pi$ is well-defined
and continuous. 

\begin{lemma}\label{surj}
The projection map $\pi:\partial^+F_{\lambda}\to \partial F$ is 
surjective. 
\end{lemma}
\begin{proof}
Assume the contrary, namely that $\pi$ would not be surjective. 
Continuous maps between compact Hausdorff spaces are closed so that 
$\pi$ is closed. Moreover each connected component of 
$\partial^+F_{\lambda}$ is sent by $\pi$ into a closed 
connected subset of $\partial F$. 

\vspace{0.2cm}
\noindent 
If some image component consists of one point then $\partial^+F_{\lambda}$ 
is a metric circle centered at that point and thus $\partial F$ has 
a point component, which  is a contradiction. 

\vspace{0.2cm}
\noindent 
Give these boundary curves the clockwise orientation.
The orientation induces a cyclic ordering on each component. Moreover,  
this cyclic order restricts to a linear order on any  proper 
subset, in particular on small neighborhoods of a point. 
When talking about left (or right) position with respect to some 
point we actually consider points which are smaller (or greater) than 
the respective point with respect to the linear order defined in a 
neighborhood of that point. 
 
 \vspace{0.2cm}
\noindent 
Let assume that some image component is a proper arc within $\partial F$. 
This arc has the right boundary point $\pi(s)$  and there is no other 
point in the image sitting to the right of $\pi(s)$, in a small neighborhood 
of $\pi(s)$. Lt $s'$ be maximal such that  
$\pi(t)=\pi(s)$ for all $t$ in the right of $s$ 
in the interval from $s$ to $s'$. As we saw above this is a proper subset 
of $\partial^+F_{\lambda}$. 

\vspace{0.2cm}
\noindent
Choose  then some 
$t\in \partial^+F_{\lambda}$  which is nearby $s'$ and slightly to 
the right of $s'$.  Therefore, we have $\pi(s)\neq \pi(t)$. 
By hypothesis $\pi(t)\in \partial F$ should sit  slightly to the 
left  and closed-by to $\pi(s)$, by  the continuity of  the map $\pi$. 

\vspace{0.2cm}
\noindent 
There are several possibilities: 
\begin{enumerate}
\item  the segments $|s\pi(s)|$ and $|t\pi(t)|$ intersect in a point $u$  
(see case 1. in the figure below). 

\vspace{0.2cm}
\noindent 
If $d_B(s,u)< d_B(t,u)$ then $d_B(s,\pi(t))\leq 
d_B(s,u)+d_B(u,\pi(t))< d_B(t,u)+d_B(u,\pi(t))=\lambda$ 
and thus $d_B(s,\partial F) <\lambda$ contradicting the fact that 
$s\in \partial^+F_{\lambda}$. 

\vspace{0.2cm}
\noindent 
If $d_B(s,u)> d_B(t,u)$ then $d_B(u, \pi(s)) < d_B(u, \pi(t)$ and hence 
$d_B(t, \pi(s))\leq d_B(t,u)+d_B(u,\pi(s)) 
< d_B(t,u)+ d_B(u, \pi(t)=\lambda$, leading to a contradiction again. 

\vspace{0.2cm}
\noindent 
Suppose now that $d_B(s,u)= d_B(t,u)$. The  previous argument shows that 
$d_B(s,\pi(t))\leq \lambda$.  In order to avoid the contradiction above 
the inequality cannot be strict, so that 
$d_B(s,\pi(t))=\lambda=d_B(s, \partial F)$.  This means  
that there are two points on $\partial F$ realizing the distance to 
$s$. This contradicts the fact that  the reach of $\partial F$ 
was  supposed to be larger than $\lambda$.  

\item The segments $|s\pi(s)|$ and $|t\pi(t)|$ have empty intersection.

\begin{enumerate}
\item Moreover, the segments $|s'\pi(s)|$ and $|t\pi(t)|$ are disjoint 
(see the case 2.a. on the figure below).

\begin{center}
\includegraphics{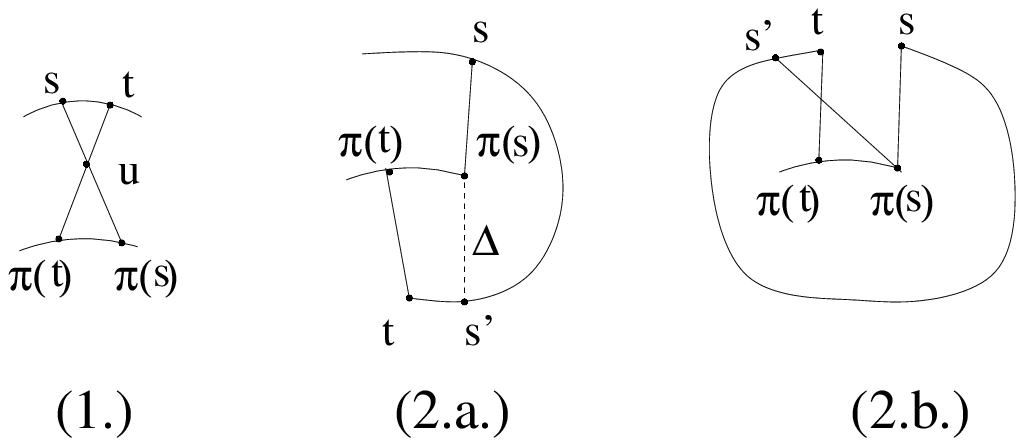}
\end{center}
  
\vspace{0.2cm}
\noindent 
In this situation we observe that the arc of metric circle $ss'$, 
the arc of $\partial F$  going clock-wisely from $\pi(t)$ to $\pi(s)$ and 
the segments $|s\pi(s)|$ and $|t\pi(t)|$ bound a domain $\Delta$  
in the plane. 
The arc of $\partial F$ which is complementary to the  clockwise arc 
$\pi(t)\pi(s)$ joins $\pi(s)$ and $\pi(t)$ and thus it has to cut 
at least once more the boundary of the domain $\Delta$. However this 
curve cannot intersect: 
\begin{enumerate}
\item  neither the arc $\pi(t)\pi(s)$, since $\partial F$ is a simple 
curve; 
\item nor the segments $|s\pi(s)|$ and $|t\pi(t)|$, since it would imply that 
there exist points in $\partial F$ at distance smaller than $\lambda$ 
on $\partial^+F_{\lambda}$. 
\item nor the arc of metric circle $ss'\subset \partial^+F_{\lambda}$, 
since the distance between $\partial^+F_{\lambda}$ and $\partial F$ 
is $\lambda >0$. 
\end{enumerate} 
Thus each alternative above leads to a contradiction. 
\item The segments $|s'\pi(s)|$ and $|t\pi(t)|$ are disjoint
(case 2.b. in the figure above). 

\vspace{0.2cm}
\noindent 
Here we conclude as in the first case by using $s'$ in the place of $s$ and  
get a contradiction again.  
\end{enumerate}
\end{enumerate}

\vspace{0.2cm}
\noindent 
Therefore our assumption was false so that the image component is 
all of $\partial F$. Notice that we actually proved that $\pi$ is open. 
\end{proof}
\begin{lemma}\label{connected}
The fibers of the projection 
map $\pi:\partial^+F_{\lambda}\to \partial F$  
are either points or  connected arcs. In particular 
$\partial^+F_{\lambda}$ is connected. 
\end{lemma}
\begin{proof}
Let $\pi(s_1)=\pi(s_2)$ for two distinct points $s_1$ and $s_2$ and 
assume that $\pi$ is not constant on the clockwise arc $s_1s_2$. 
Pick up some $v$ in the arc $s_1s_2$.  
According to the proof of the previous lemma 
we cannot have $\pi(v)$ sitting to the left of $\pi(s_1)$, for $v$ near $s_1$. 
Thus $\pi(v)$ sits in the right of $\pi_(s_1)$. 
Moreover, if $w$ lies between $v$ and $s_2$ the same argument shows that 
$\pi(w)$ sits in the right of $\pi(v)$. Consequently the 
image by $\pi$ of the arc  $s_1s_2$ covers completely $\partial F$ and 
the situation is that from the figure below. 

\begin{center}
\includegraphics{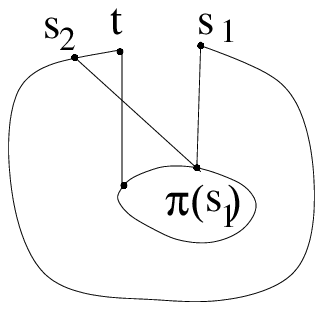}
\end{center}

\vspace{0.2cm}
\noindent 
Take now any $t$ in the  complementary arc $s_2s_1$. 
If $\pi(t)\neq\pi(s_1)$ then $|t\pi(t)|$ intersects either 
$|s_1\pi(s_1)|$ or else $|s_2\pi(s_2)|$, leading to a contradiction   
as in the proof of the previous lemma. The lemma follows. 
\end{proof}

\vspace{0.2cm}
\noindent 
We will need to have informations about the rectifiability of 
the set $\partial^+F_{\lambda}$, as follows: 

\begin{lemma}\label{curve}
If $0< \lambda< r(\partial F)$ then 
$\partial^+F_{\lambda}$ is a Lipschitz curve and in particular a 
${\mathcal C}^{1,1}$ simple closed curve. 
\end{lemma}
\begin{proof} 
Since $\lambda$ is smaller than the reach $r(\partial F)$ it follows that 
$\partial^+F_{\lambda}$ has also positive reach. 
The proof from \cite{Fe} shows that $\partial^+F_{\lambda}$ is 
a 1-manifold. Thus, by Theorem \ref{lip} the set $\partial^+F_{\lambda}$ is 
a Lipschitz 1-manifold. Lemma \ref{connected} shows that 
$\partial^+F_{\lambda}$ is connected and thus it is a simple closed 
curve. 
\end{proof}

\vspace{0.2cm}
\noindent 
Therefore the curve $\partial^+F_{\lambda}$ is  rectifiable.
Recall that  $\partial^+F_{\lambda}$ has an orientation, say the clockwise one.
Consider a maximal $\lambda$-necklace $\mathcal B$ and suppose that there 
exists a pair of consecutive beads which do not touch each other. 
There is induced a cyclic order on the beads 
of any $\lambda$-necklace: the beads $B_1$, $B_2$ and $B_3$ are  
cyclically ordered if the three corresponding points on which the 
$B_i$ touch $\partial F$ are cyclically ordered. 
As $\lambda < r(\partial F)$ each $\lambda$-bead intersects $\partial F$ 
in a unique point and thus the definition makes sense.   

\vspace{0.2cm}
\noindent 
Consider two consecutive beads which do not touch each other. 
If $x\in \partial F$ let $l_x$ be some support line for $\partial F$ at 
$x$  and $B_x$ (depending also on $l_x$) the translate of $\lambda B$ 
which admits $l_x$ as support line at $x$. We assume that going from 
$x$ to the center of $B_x$ we go locally outward $F$. 
We call $B_x$ the virtual $\lambda$-bead attached at $x$. 
Actually the virtual bead might intersect $\partial F$ and thus 
be not a bead. 

\vspace{0.2cm}
\noindent 
The consecutive beads are $B_p$  and $B_q$ for $p,q\in \partial F$. 
We want to slide $B_q$  in counterclockwise direction among the virtual beads 
$B_x$, where $x$ is going from $q$ to $p$ along $\partial F$ until 
$B_x$ touches $B_p$. 
Let ${\mathcal B}_x$ be the virtual necklace obtained from 
the necklace $\mathcal B$ by replacing the bead $B_q$ by the virtual bead  
$B_x$. 

\vspace{0.2cm}
\noindent 
If all virtual necklaces ${\mathcal B}_x$ are genuine necklaces then 
we obtained another maximal necklace in which the  
pair of consecutive beads are now touching each other. We continue this 
procedure while possible. Eventually we stop either when the necklace was 
transformed into an almost complete one, or else the sliding procedure 
cannot be performed anymore. 

\vspace{0.2cm}
\noindent 
Let then assume we have two consecutive beads which cannot get closer 
by sliding. Let then $a$ be the first point on the curve segment from 
$q$ to $p$ (running counter-clockwisely) where the sliding procedure 
gets stalked. We have then two possibilities: 
\begin{enumerate}
\item $B_a$ touches $\partial F$ in one more point. 
\item $B_a$ touches another bead $B_b$ from the necklace $\mathcal B$. 
\end{enumerate}  

\vspace{0.2cm}
\noindent 
In the first situation the center of $B_a$ is at distance $\lambda$ 
from $\partial F$ and the distance is realized twice. Thus 
$r(\partial F) \leq  \lambda$, contradicting our choice of $\lambda$. 

\vspace{0.2cm}
\noindent 
The analysis of the second alternative is slightly more delicate.
Let $z$ be the midpoint of the segment $|xy|$ joining the 
centers of the two beads $B_a$ and $B_b$ respectively.

\vspace{0.2cm}
\begin{center}
\includegraphics{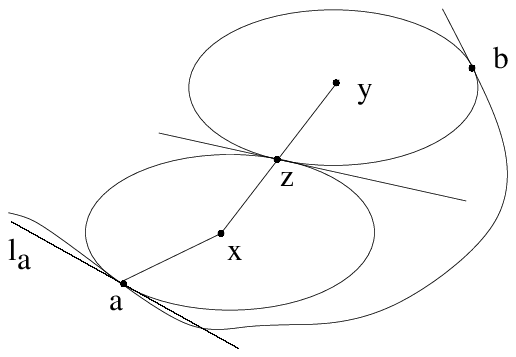}
\end{center}

\vspace{0.2cm}
\noindent 
Let $l_z$ be a support line at $z$, common to both $B_a$ and $B_b$. 
\begin{lemma}\label{angle} 
Either $l_a$ and $l_z$ are parallel or else they intersect in the half-plane 
determined by $xy$ and containing the germ of  the arc of $\partial F$ 
issued from  $a$ which goes toward $p$. 
\end{lemma}
\begin{proof}
Assume the contrary and let then $l_w$ be a support line 
to $B_a$ which is parallel to $l_z$ and touches $\partial B_a$ into the 
point $w\in \partial B_a$. The cyclic order on $\partial B_a$ is then 
$z,a$ and $w$. Consider the arc of $\partial F$ issued from $a$.
Since the reach of $\partial F$ is larger than $\lambda$ 
we have $w$ and all points of $B_a\setminus\{a\}$ are contained in 
$\R^2-F$. Thus there is some $\varepsilon$-neighborhood of $w$ which is 
still contained in the open set $\R^2-F$. This implies that we can translate 
slightly $B_a$ along $l_z$ within the strip determined by $l_z$ and $l_w$ 
such that it does not intersect $\partial F$ anymore. 

\vspace{0.2cm}
\begin{center}
\includegraphics{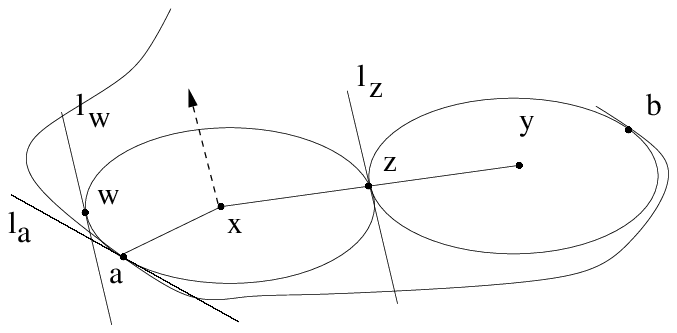}
\end{center}

\vspace{0.2cm}
\noindent 
The translated $B_a$ will remain disjoint from ${\rm int}(B_b)$ because 
the later lies in the other half-plane determined by $l_z$. 
Pushing it further towards $F$ along $l_a$ we find that the sliding 
can be pursued beyond $a$, contradicting our choice for $a$. 
This proves the claim. 
\end{proof}
 
\begin{lemma}
For any $t\in |xy|$ we have $d_B(t, \partial F)\leq 2\lambda$. 
\end{lemma}
\begin{proof}
The segment $|xy|$ is covered by $B_a\cup B_b$ and the triangle inequality 
shows that $\min(d_B(t,a), d_B(t,b))\leq 2\lambda$, which implies the claim.   
\end{proof}
 
\vspace{0.2cm}
\noindent 
Consider now $\partial^+F_{2\lambda}$. By lemmas \ref{surj} and 
\ref{curve} the projection $\pi:\partial^+F_{2\lambda}\to \partial F$ 
is a surjection. 
Let us choose some $w\in \partial^+F_{2\lambda}$ such that 
$\pi(w)=a$. Set $x'$ for the midpoint of the segment $|aw|$. 
\begin{lemma}
The metric ball $B(x',\lambda)$ is a $\lambda$-bead.
\end{lemma} 
\begin{proof}
As $a\in B(x',\lambda)\cap \partial F$ it suffices to show that 
$B(x', \lambda)\subset \R^2\setminus {\rm int}(F)$. 
Suppose the contrary and let $p\in {\rm int}(B(x',\lambda))\cap {\rm int}(F)$. 
There exists then some $p'\in |px|\setminus \{p\}$ with 
$p'\in B(x',\lambda)\cap \partial F$. 
The diameter of $B(x', \lambda)$ is $2\lambda$ and so
$d_B(w,p)\leq 2\lambda$, but $p'$ lies on the segment $|pw|$ 
so that $d_B(p',w) <2\lambda$. This implies that $d_B(w,\partial F) <2\lambda$ 
which is a contradiction.  
This establishes the lemma. 
\end{proof}

 \vspace{0.2cm}
\noindent 
The diameter of a $\lambda$-bead is obviously $2\lambda$. 
We say that points $u$ and $v$ are {\em opposite} points in the bead 
if they realize the diameter of the bead. If $B$ is strictly convex the 
each boundary point has a unique opposite point. This is not anymore true 
in general. Given a point on the boundary of a rectangle 
any point on the opposite side is an opposite of the former one.

\begin{lemma}
There exists some point $w$ which is opposite to $a$ in $B_a$ such that 
$w\in \partial^+ F_{2\lambda}$. 
\end{lemma}
\begin{proof}
Let us assume first that $\partial F$ is smooth at $a$, or equivalently that 
it has unique support line at $a$. 
As both $B(x',\lambda)$ and $B_a$ are $\lambda$-beads which have 
the same support line $l_a$ (since it is unique) it follows that 
they coincide. In other terms $w$ is one of the points opposite 
to $a$ in $B_a$. 

\vspace{0.2cm}
\noindent 
Consider now the general case when $\partial F$ is not necessarily smooth 
at $a$. Let $l_a^+$ and $l_a^-$ denote the extreme positions of the 
support lines to $\partial F$ at $a$. Thus $l_a$ belongs to the cone 
determined by $l_a^+$ and $l_a^-$. 
Recall that $\partial F$ was supposed to be Lipschitz and thus by 
the Rademacher theorem it is almost everywhere differentiable. 
There exists  then a sequence of points  $p_j^{\pm}\in \partial F$ 
converging to $a$ such that 
$\partial F$ is smooth at $p_j^+$ and $p_j^-$ and   
the tangent lines at $p_j^+$ (respectively $p_j^-$) 
converge to $l_a^+$ (respectively to $l_a^-$). 

\vspace{0.2cm}
\noindent 
Let $w_j^{\pm}$ be points on $\partial^+F_{2\lambda}$ such that 
$\pi(w_j^{\pm})=p_j^{\pm}$. It follows that $w_j^{+}$ (respectively 
$w_j^{-}$) converge towards a point $w^+$  (respectively $w^-$) 
which lies on the  boundary of a $\lambda$-bead $B(x^+,\lambda)$ (respectively 
$B(x^-,\lambda)$ having the support line $l_a^+$ (respectively $l_a^-$). 
Further $\pi(w^+)=\pi(w^-)=a$. The proof of Lemma \ref{connected} shows that 
the arc of the metric circle centered at $a$ and of radius $2\lambda$ 
which joins $w^+$ to $w^-$ is also contained in $\partial^+F_{2\lambda}$. 
The point $w$ which is opposite   to $a$ in the 
$\lambda$-bead $B_a$ is contained in this arc and thus it belongs to 
$\partial^+F_{2\lambda}$.

\end{proof}
\vspace{0.2cm}
\noindent 
{\em End of the proof of Proposition \ref{necklace}}. 
The clockwise arc $ab$ of $\partial F$ and the union of segments 
$|ax|\cup |xy|\cup |yb|$ which is disjoint from $\partial F$ bound 
together a simply connected domain $\Omega_0$ in the plane. 
Then $\Omega=\Omega_0\setminus {\rm int}(B_a)\cup {\rm int}(B_b)$ is also 
a topological disk, possibly with an arc attached to it (if 
$B_a\cap B_b$ is an arc)  
since it is obtained from $\Omega_0$ by deleting out two small disks 
touching the boundary and having connected intersection.

\begin{center}
\includegraphics{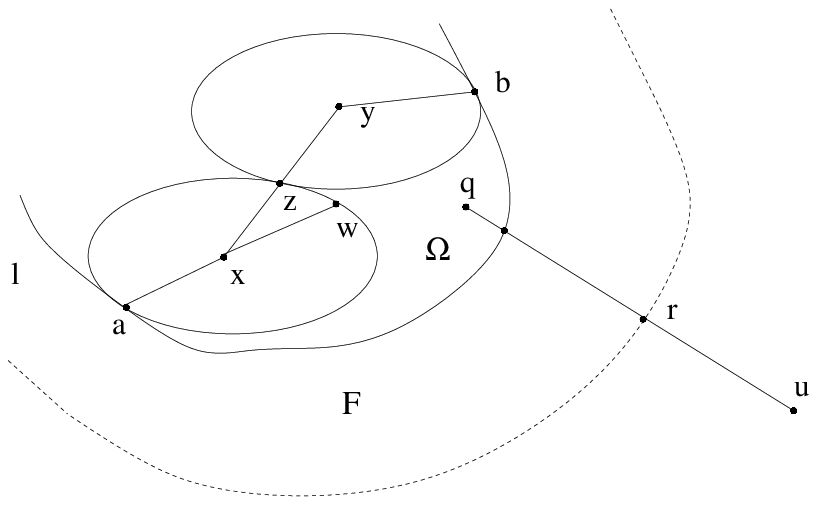}
\end{center} 
According to lemma \ref{angle} $w$ belongs either to $\Omega$  
(for instance when $B$ is strictly convex) or else 
to $B_a\cap B_b$ (when the  support line $l_a$ meets $B_a$ along a segment).

\vspace{0.2cm}
\noindent 
On the other hand the curve $\partial^+F_{2\lambda}$  contains both the 
point $w\in \Omega$ and points outside $\Omega$. 
In fact the arc $ab$ is contained in $\partial F$ 
which bounds the domain $F$. Pick up a point  $q$ of $\Omega$ and 
$r$ in the arc $ab$ such that the half-line $|vr$ does not meet 
$|ax|\cup |xy|\cup |yb|$. Then $|vr$ intersects the domain 
$F$ and thus at least once the clockwise arc $ba$. Let $r$ be such a 
point. Then there exists $u\in \partial^+F_{2\lambda}$ for which 
$\pi(u)=r$. It is clear that $u\not\in \Omega$.  Otherwise, by Jordan 
curve theorem the segment $|ru|$ should intersect once more the 
clockwise arc $ab$ and this would contradict the fact that 
$d_B(u,r)=d_B(u, \partial F)$.

\vspace{0.2cm}
\noindent 
Therefore  the curve $\partial^+F_{2\lambda}$ has to 
exit the domain $\Omega$ and there are two possibilities: 
\begin{enumerate}
\item either $\partial^+F_{2\lambda}$ meets ${\rm int}(B_a)\cup 
{\rm int}(B_b)$.  This will furnish points of $\partial^+F_{2\lambda}$ at 
distance less than $2\lambda$ from  either $a$ or $b$ and thus 
from $\partial F$ and hence it leads to a contradiction. 
\item or else $\partial^+F_{2\lambda}$ meets $\partial B_a\cup\partial B_b$.
In this case any point from  $(\partial B_a\cup\partial B_b)\cap \partial^+F_{2\lambda}$ is at distance $2\lambda$ both from $a$ and from $b$. 
In particular the distance $2\lambda$ is not uniquely realized and 
this contradicts the choice of $2\lambda < r(\partial F)$. 
\end{enumerate}

\subsection{Proof of Proposition \ref{gap}}
Assume that $d_B(a,c)\geq 5\lambda$.   
Let $d\in |ac|$ be the midpoint of $|ac|$ and $D$ denote 
the translate of $A$ centered at $d$. The triangle inequality shows that 
$A\cap D=C\cap D=\emptyset$. 
The segment $|ac|$ intersects once each one of $A$ and $C$. Consider 
the support lines $l_A$ and $l_C$ at these points. Since $A$ and $C$ 
are obtained by a translation one from the other, we can choose 
the support lines to be parallel. The convexity of $D$ implies that 
$D$ is contained in the strip $S$ determined by the parallel lines 
$l_A$ and $l_C$.

\vspace{0.2cm}
\begin{center}
\includegraphics{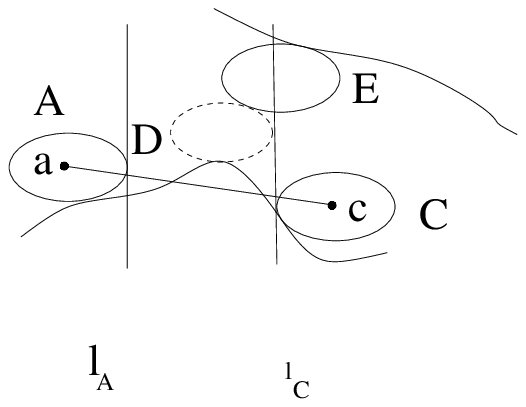}
\end{center}
\vspace{0.2cm}

\vspace{0.2cm}
\noindent 
If  $D\cap \partial F$ is empty, then we translate it within $S$ 
until it touches first $\partial F$. If $D$ intersects 
non-trivially the interior of $F$ on one side of the segment $|ac|$, then we 
translate it in the opposite direction until the contact between $D$ and 
$F$ is along boundary points. We keep the notation $D$ for the 
translated oval.  However,  
by the maximality  of our almost complete $\lambda$-necklace, 
we cannot add $D$ to our beads to make a necklace.  Thus $D$ has to 
intersect either once more  
$\partial F$,  or else another bead $E$  from the necklace. 

\vspace{0.2cm}
\noindent 
Let consider the first situation. 
We deflate gradually the  bead $D$  
by a homothety of ratio going from $1$ to $0$ by keeping its boundary 
contact with $\partial F$ until we reach a position where 
all contact points between $D$ and $F$ are boundary  
points in $\partial F$. This implies that the reach of $\partial F$ 
is less than $\lambda$ which contradicts our choice of $\lambda$. 

\vspace{0.2cm}
\noindent 
When the second alternative holds true we make use of the following: 
\begin{lemma}\label{consecutive}
If two $\lambda$-beads intersect each other and there exist 
$\lambda$-beads between them (both in the clockwise and the 
counterclockwise directions) then the reach of $\partial F$ is at most 
$2\lambda$. 
\end{lemma} 
\begin{proof}
If $D$ and $E$ are the two beads which intersect non-trivially at $z$  
let $d$ and $e$ be the points where they touch $\partial F$. 
One can choose one of the arcs $de$ or $ed$ of $\partial F$ such that 
together with $|dz|\cup |ze|$ bound a simply connected bounded 
domain $\Omega_0$ which is disjoint from $F$.
There exists at least one other $\lambda$-bead say $G\subset \Omega_0$. 
Then we can find as above a point $w\in \partial G$ which lies in 
$\partial^+F_{2\lambda}$. Therefore $\partial^+F_{2\lambda}$ contains 
points from $\Omega_0$.  It is not hard to see that the argument given 
at the end of the proof of Proposition \ref{necklace} shows that 
$\partial^+F_{2\lambda}$ has also points from outside $\Omega_0$. 
However $\partial^+F_{2\lambda}$ is connected and disjoint from 
$\partial F$ and hence it has to cross $D\cup E$. 
But then we will find that either there are points on $\partial^+F_{2\lambda}$ 
of distance $2\lambda$ from both $d$ and $e$ (contradicting the fact that 
the reach was larger than $2\lambda$) or else we 
find point at distance strictly less than $2\lambda$ from 
either $d$ or $e$, which contradicts the definition of 
$\partial^+F_{2\lambda}$.   
\end{proof}

\vspace{0.2cm}
\noindent 
In our case both $A$ and $C$ are $\lambda$-bead disjoint from $D$ 
both in clockwise and counterclockwise directions. Thus if 
$D$ intersects another bead  $E$, different from $A$ and $C$,  
of the necklace then the reach of $\partial F$ will be smaller than 
$2\lambda$. This contradiction shows that we can  add $D$ to our necklace 
and the Proposition \ref{gap} is proved.

\subsection{Proof of Proposition \ref{leng}}
We will prove first the Proposition \ref{leng} in the case when 
$\partial F$ is a polygon $Q$. 
Denote by $Q_{\lambda}$ the set of points lying outside $Q$ and 
having distance $\lambda$ to $Q$ (or, this is the same, to $\partial Q$).
Let us define a (not necessarily simple) curve $W_{\lambda}$ as follows. 
To each edge $e$ of $Q$ there is associated a parallel segment 
$e_{\lambda}$ which is the translation of $e$ in  outward (with respect to $Q$)
direction dual to $e$.  

\vspace{0.2cm}
\noindent 
Recall the definition of the dual to a given direction. 
Assume for the moment that $\partial B$ is strictly convex. 
If $d$ is a line then let $d_+$ and $d_-$ be support lines to $\partial B$ 
which are parallel to $d$; by the strict convexity assumption 
each lines $d_+,d_-$ intersects $\partial B$ into one point 
$p_+, p_-$ respectively. Then the dual of $d$ is the line $p_+p_-$ (which 
passes through the origin). If $\partial B$ is not strict convex then 
it might still happen that each support line parallel to $d$ has 
one intersection point with $\partial B$, in which case the definition 
of the dual is the same as above. Otherwise $d_+\cap \partial B$ has 
at least two points and thus, by convexity, it should be a segment $z_+t_+$. 
In a similar way $d_-\cap \partial B$ is the a segment 
$z_-t_-$ which is the symmetric of $z_+t_+$ with respect to the 
center of $B$. Thus $z_+z_-t_-t_+$ is a parallelogram having two sides 
parallel to $d$. The direction of the other two sides is the dual of $d$. 

\vspace{0.2cm}
\noindent 
It is immediate then that $d_B(e, e_{\lambda})=\lambda$.

\vspace{0.2cm}
\noindent 
For each vertex $v$ of $Q$ where the edges $e$ and $f$ meet together 
we will associate an arc $v_{\lambda}$ of the circle $\lambda \partial B$ of radius $\lambda$. Let $n_{e}$ and $n_{f}$ be the length $\lambda$  
vectors  whose directions are dual to $e$ and $f$ respectively  
and are pointing outward $Q$. 
Let $v_{\lambda}$ be 
the arc of $\lambda B$ corresponding to the trajectory drawn by $n_e$ 
when rotated to arrive in position $n_f$ while pointing outward of $Q$. 

\vspace{0.2cm}
\noindent 
Let us order cyclically the edges $e_1,e_2,\ldots, e_n$ of $Q$ clockwisely and 
the vertices $v_j$ (which is common to $e_j$ and $e_{j+1}$).  
Let also $A_j$ (respectively $B_j$) denote the left (respectively right) 
endpoint of ${e_j}_{\lambda}$. Set $\alpha_j$ for the interior angle 
(with respect to $Q$) between $e_j$ and $e_{j+1}$. 
Observe that the configuration around two consecutive edges is one of the following type: 
\begin{enumerate}
\item if $\alpha_j\leq \pi$ then ${e_j}_{\lambda}$ and ${e_{j+1}}_{\lambda}$ 
are disjoint and joined by the arc ${v_j}_{\lambda}$ is 
which is locally outside $Q$; 
\item if $\pi< \alpha_j< 2\pi$ then 
${e_j}_{\lambda}$ and ${e_{j+1}}_{\lambda}$ intersect at some point $C_j$.
\end{enumerate}

\vspace{0.2cm}
\noindent 
Let us define ${e_j}^*_{\lambda}$ to be the segment whose left endpoint 
is $C_{j-1}$, if $\alpha_{j-1} > \pi$ and $A_j$  elsewhere 
while the right endpoint 
is $C_j$, if $\alpha_j >\pi$, and $B_j$ otherwise. 
Let also ${v_j^*}_{\lambda}$ be 
empty when $\alpha_j >\pi$ and the arc ${v_j}_{\lambda}$ otherwise. 

\vspace{0.2cm}
\noindent 
Set $W_{\lambda}$ for the union of edges ${e_j^*}_{\lambda}$ 
and of arcs ${v_j^*}_{\lambda}$.  
Notice that $W_{\lambda}$ might have (global) self-intersections. 

\vspace{0.2cm}
\noindent 
Observe that $Q_{\lambda}\subset W_{\lambda}$. Notice that the inclusion
might be proper. 

\vspace{0.2cm}
\noindent 
We claim now that: 
\begin{lemma}
 The length  of $W_{\lambda}$ verifies 
\begin{equation} p_B(W_{\lambda})\leq p_B(Q)+\lambda p_B(\partial B)\end{equation}
\end{lemma}
\begin{proof}
The arcs ${v_j}_{\lambda}$ are naturally oriented, using the orientation of 
$\partial Q$. Moreover, its orientation is positive if $\alpha_j \leq \pi$ and 
negative otherwise.  
Since $\sum_{j=1}^n \alpha_j=(n-2)\pi$ we have 
$\sum_{j=1}^n (\pi- \alpha_j)=2\pi$, which means that the 
algebraic sum of the arcs ${v_{j}}_{\lambda}$ is once the circumference 
of $\lambda \partial B$. Thus 
\begin{equation}  \sum_{j=1}^n \sigma(v_j)p_B({v_{j}}_{\lambda})=\lambda p_B(\partial B)\end{equation}
where $\sigma_j\in\{-1,1\}$ is the sign giving the orientation of 
${v_{j}}_{\lambda}$. 
It follows that 
\begin{equation}  \lambda p_B(\partial B)+p_B(Q) = 
\sum_{j=1}^n \sigma(v_j)p_B({v_{j}}_{\lambda}) + 
\sum_{j=1}^n p_B(e_j)\end{equation}

\vspace{0.2cm}
\noindent 
Now,  $\sigma(j)=-1$ if and only if $C_j$ is defined (i.e. the 
angle $\alpha_j >\pi$). Thus 
\begin{eqnarray} \sum_{j=1}^n \sigma(v_j)p_B({v_{j}}_{\lambda}) + 
\sum_{j=1}^n p_B(e_j) & = &
\sum_{j=1; \alpha_j \leq \pi}^n (p_B({v_{j}}_{\lambda}) + p_B(e_j)) + \nonumber \\
& & + \sum_{j=1; \alpha_j > \pi}^n (|A_{j+1}C_j|_B + |C_jB_j|_B - 
p_B({v_j}_{\lambda}) + p_B({e_j^*}_{\lambda}))=\nonumber \\ 
&  = & p_B(W_{\lambda})+  \sum_{j=1; \alpha_j > \pi}^n (|A_{j+1}C_j|_B + |C_jB_j|_B - p_B({v_{j}}_{\lambda}) \geq  \nonumber \\
& \geq & p_B(W_{\lambda})\end{eqnarray}
The last inequality follows from 
\begin{equation}  |A_{j+1}C_j|_B + |C_jB_j|_B \geq p_B({v_{j}}_{\lambda})\end{equation}
In fact, it is proved in (\cite{T}, p.121),  see also  
or the elementary proof from (\cite{MSW1}, 3.4., p.111-113), that 
 a convex curve  is 
shorter than any other curve surrounding it. 
Moreover the direction $B_jC_j$ is dual to $e_j$ and thus it is 
tangent to a copy of $\lambda B$ translated at $v_j$; in a similar way 
$A_{j+1}C_j$ is dual to $e_{j+1}$ and thus tangent to the same 
copy of $\lambda B$. In other words the arc ${v_j}_{\lambda}$ determined 
by $A_{j+1}$ and $B_j$ is surrounded by the union $|A_{j+1}C_j|\cup |C_jB_j|$ 
of two support segments. The convexity of $\partial B$ implies the 
inequality above, and in particular our claim. 
\end{proof}

\begin{remark}
One can use the signed measures  defined by Stach\'o in \cite{St2} 
for computing the length of $\partial^+F_{\lambda}$ 
and to obtain, as a corollary, the result of Proposition \ref{leng}.
Our proof for planar rectifiable curves  has the advantage to 
be completely elementary. 
\end{remark}

\vspace{0.2cm}
\noindent {\em End of the proof of proposition \ref{leng}}.
Let now $\partial F$ be an arbitrary rectifiable simple curve. 
It is known that there exists a sequence of polygons $Q_n$ inscribed 
in $\partial F$ such that 
$\lim_n p_B(Q_n)=p_B(\partial F)$. 
Here $p_B$ denotes the  Jordan (equivalently Lebesgue) length of the 
respective curve, in the Minkowski metric. 

\vspace{0.2cm}
\noindent 
Therefore ${Q_n}_{\lambda}$ is a sequence of rectifiable curves  
which  converge to $\partial^+F_{\lambda}$. 
By theorem \ref{lip} $\partial^+F_{\lambda}$ is the union of finitely many 
Lipschitz 1-manifolds and thus the Lebesgue length of $\partial^+F_{\lambda}$ 
makes sense. By the lower semi-continuity 
of the Lebesgue length (see e.g. \cite{Ce}) it follows that 
\begin{equation} \lim_n\inf p_B({Q_n}_{\lambda}) \geq p_B(\partial^+F_{\lambda})\end{equation}
However we proved above that for simple polygonal lines ${Q_n}$ 
we have: 
\begin{equation}  p_B({Q_n}_{\lambda}) \leq p_B(Q_n) + \lambda p_B(\partial B)\end{equation}
Passing to the limit  $n \to \infty$ we obtain 
\begin{equation}   p_B(\partial^+F_{\lambda}) \leq \lim_n\inf p_B(Q_n) + \lambda p_B(\partial B)=
p_B(\partial F) + \lambda p_B(\partial B)\end{equation}
Therefore Proposition \ref{leng} follows.

\subsection{Curves of  zero reach}

Consider now an arbitrary simple closed Lipschitz curve $\partial F$ 
in the plane. When sliding  $\lambda$-beads for achieving almost completeness 
of necklaces  we  might get stalked because we 
encounter points of $\partial F$ with reach smaller than $\lambda$. 
Let us introduce the following definitions. 
 
\begin{definition}
The clockwise arc $ab$ of $\partial F$ is a $\lambda$-corner 
if there exists a $\lambda$-bead $B_a$ which touches $\partial F$ at $a$ and 
$b$ and such that there is no $\lambda$-bead $B_x$ for $x$ in the interior 
of the arc $ab$ (except possibly for $B_a$). 
\end{definition}

\begin{definition}
The clockwise arcs $aa'$ and $b'b$ of $\partial F$ form a 
{\em long} $\lambda$-gallery if there exist two 
{\em disjoint} $\lambda$-beads $B_a$ and $B_{a'}$ with 
$\{a,b\}\subset B_a\cap \partial F$ and 
 $\{a',b'\}\subset B_{a'}\cap \partial F$ such that: 
\begin{enumerate}
\item there is no $\lambda$-bead touching the arcs $aa'$ or $bb'$; 
\item at least one  complementary arc among $a'b'$ and $ba$ 
admits a $\lambda$-bead which is disjoint from $B_a$ and $B_{a'}$. 
\end{enumerate}
\end{definition}

\begin{definition}
The clockwise arcs $aa'$ and $b'b$ of $\partial F$ form a 
{\em short} $\lambda$-gallery if there exist two 
$\lambda$-beads $B_a$ and $B_{a'}$ with non-empty intersection,  
$\{a,b\}\subset B_a\cap \partial F$ and 
 $\{a',b'\}\subset B_{a'}\cap \partial F$ such that: 
\begin{enumerate}
\item any $\lambda$-bead touching $aa'\cup b'b$ should intersect 
the boundary beads $B_a\cup B_{a'}$; 
\item there is no $2\lambda$-bead touching the arcs $aa'\cup b'b$; 
\end{enumerate}
\end{definition}

\vspace{0.2cm}
\noindent 
Observe that $\lambda$-corners do not really make problems in sliding 
$\lambda$-beads, because we can jump from $a$ to $b$ keeping the same bead and 
we can continue the sliding from there on.

\vspace{0.2cm}
\noindent 
Set $Z_{\lambda}$ for the set of points that belong to some 
$\lambda$-gallery (long or short).

\begin{lemma}
For each $\lambda >0$ the number of maximal 
$\lambda$-galleries is finite. 
\end{lemma}
\begin{proof}
Assume that we have infinitely many $\lambda$-galleries. 
They have to be disjoint, except possibly for their boundary points. 
Thus the length of their arcs converges to zero. 
Moreover, the associated pairs of arcs of $\partial F$ 
converge towards a pair of two points at distance $2\lambda$. 
Thus all but finitely many galleries are short galleries. 
The lengths of intermediary arcs (those joining consecutive gallery arcs 
in the sequence) should have their length going to zero since 
their total length is finite. 

\vspace{0.2cm}
\noindent 
Consider now the union of two consecutive galleries in the sequence 
together with the intermediary arcs between them. We claim that 
if we are deep enough in the sequence then this union will also be 
a gallery, thus contradicting the maximality. 
Assume the contrary, namely that the union is not a short gallery. 
Then one should find  either a $\lambda$-bead touching one intermediary arc
which is disjoint from the boundary beads, or else a $2\lambda$-bead.

\vspace{0.2cm}
\noindent 
In the first case the intermediary arc joins two points $x,y$ of 
intersecting $\lambda$-beads and surrounds a disjoint $\lambda$-bead. 
Let $z$ be a common point for the two boundary beads. Then 
the union of $|xz|\cup |zy|$ with the intermediary arc forms a closed 
curve surrounding the  boundary of a $\lambda$-bead. In particular its length 
is larger than or equal to $\lambda p_B(\partial B)$. Since 
$d_B(x,z), d_B(z,y)\leq 2\lambda$ it follows that the length of the 
intermediary arc is at least $\lambda (p_B(\partial B)-4)\geq 2\lambda$. 
However intermediary arcs should have length going to zero, so this 
is a contradiction. 

\vspace{0.2cm}
\noindent 
The second alternative tells us that there exists a $2\lambda$-bead touching 
the intermediary arc. Let the arcs be $xy$ and $x'y'$. Then we claim that 
the union of the arcs $xx'$ (in the boundary of the bead), $x'y'$, $y'y$ 
(in the boundary of the bead) and $yx$ is a closed curve surrounding 
the convex $2\lambda$-bead.  
Therefore their total length is at least 
$2\lambda p_B(\partial B)$. 

\vspace{0.2cm}
\noindent 
In fact suppose that the $2\lambda$-bead of center $w$ intersects 
the arc $xx'$. Observe that $w$ is  not contained in 
the  interior of the $\lambda$-bead because otherwise the 
$2\lambda$-bead would contain it and thus there will be no place for the arc 
of $\partial F$. Further we find that the distance function 
$d_B(z,w)$ for $z$ in the arc $xx'$ will have points where it takes values 
smaller than $2\lambda$. As $d_B(x,w), d_B(x',w)\geq 2\lambda$ it follows 
that the distance function will have  at least two local maxima. However 
since $B$ is convex the distance function to a point cannot have several 
local maxima unless when $B$ is not strictly convex and there is a 
segment of maxima. This proves the claim.  
 
\vspace{0.2cm}
\noindent 
\begin{center}
\includegraphics{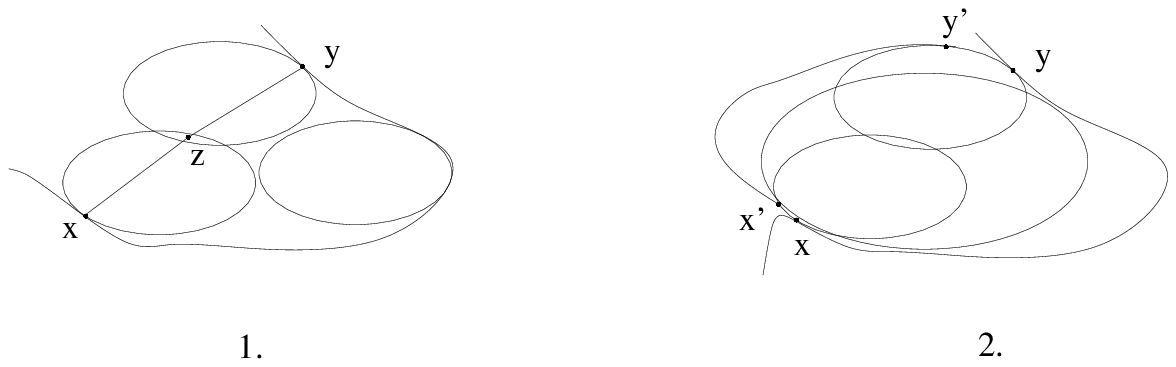}
\end{center}

\vspace{0.2cm}
\noindent 
However the  sum of the lengths of the arcs $xx'$ and $y'y$ 
is smaller than $\frac{4\lambda}{3}p_B(\partial B)$ if we are far enough in 
the sequence. 
Indeed the two boundary $\lambda$-beads intersect each other and 
their centers become closer and closer as we approach the limit bead. 
Then the perimeter of the union of the two convex $\lambda$-beads converge 
to the perimeter of one bead. In particular, at some point it becomes smaller 
than $\frac{4\lambda}{3}p_B(\partial B)$.   

\vspace{0.2cm}
\noindent 
This implies that the length of the arcs $xy$ and $x'y'$ is at least 
 $\frac{2\lambda}{3}p_B(\partial B)\geq 4\lambda$. This is in contradiction with  
the fact that intermediary arcs should converge to points. 
\end{proof}

\begin{lemma}\label{liminf}
For any $\lambda >0$ we have 
\begin{equation} p_B(\partial F\setminus Z_{2\lambda}\cup Z_{\lambda})\leq \lim_{\delta\to 0}2\delta N_{\delta}(\partial F)\end{equation}
\end{lemma}
\begin{proof}
Since there are finitely many  maximal $2\lambda$-galleries consider 
$\delta$ be small enough such that two $\delta$-beads which 
touch a maximal $2\lambda$-gallery at each end point should be disjoint. 

\vspace{0.2cm}
\noindent 
Let then choose a $\delta$-necklace ${\mathcal N}_j$ for each connected 
component  $A_j$ of $\partial F\setminus Z_{2\lambda}$. 
We claim that the union of necklaces $\cup_j{\mathcal N}$ 
is a necklace on $\partial F$.
  
\vspace{0.2cm}
\noindent 
No bead exterior to a $2\lambda$-gallery can intersect the arcs 
of that gallery. Extreme positions of $\delta$-beads 
are contained in boundary beads and thus only boundary points of 
the gallery can be touched by the necklace. 

 \vspace{0.2cm}
\noindent 
Two component necklaces are separated by a gallery. We chose $\delta$ such that the  last $\delta$-bead of one necklace is 
disjoint from the first $\delta$-bead of the next component. 

\vspace{0.2cm}
\noindent 
Remark now that $\delta$-necklaces with $\delta <\lambda$  
are otherwise disjoint. 
In fact suppose that two beads from different necklaces (or one bead 
from a necklace and an arc $A_j$) intersect each other. 
Going far enough to one side of the arcs we should find large enough 
beads and hence $2\lambda$-beads,  
since beads lie in $\R^2\setminus F$. Going to  the other 
side, if we can find a $2\lambda$-beads then the two arcs contain a 
$2\lambda$-gallery, contradicting our assumptions. Otherwise the 
remaining part forms a $2\lambda$-corner and in particular the arcs 
belong to the same component.  Then the beads should be disjoint, since 
they are beads of the same necklace. 
The same proof works for the bead intersecting an arc.

\vspace{0.2cm}
\noindent 
Let then $N_{\delta}(A_j)$ be the 
maximal cardinal of a $\delta$-necklace  in 
$\R^2- F$ such that all beads touch the arc $A_j$. We set (by abuse of 
notation)
$N_{\delta}(\partial F\setminus Z_{2\lambda}\cup Z_{\lambda})=
\sum_jN_{\delta}(A_j)$. 

\vspace{0.2cm}
\noindent 
Summing up we proved above that 
\begin{equation}  N_{\delta}(\partial F\setminus Z_{2\lambda}\cup Z_{\lambda})\leq N_{\delta}(\partial F)\end{equation}

\vspace{0.2cm}
\noindent 
Recall now that each arc $A_j$ is  a Lipschitz 
curve of reach at least
$2\lambda$. The proof of Proposition \ref{limreach} can be carried over 
not only for simple closed curves but also for simple Lipschitz 
arcs of positive reach without essential modifications, 
with a slightly different upper bound in 
Proposition \ref{leng}.
 
\vspace{0.2cm}
\noindent 
Thus the result holds true for each one of the arcs 
$A_j$. As we have finitely many such arcs $A_j$  we obtain  
\begin{equation}   \lim_{\delta\to 0}2\delta 
N_{\delta}(\partial F\setminus Z_{2\lambda}\cup Z_{\lambda})= \sum_jp_B(A_j)
=p_B(\partial F\setminus Z_{\lambda})\end{equation}
The inequality above implies the one from the statement. 
\end{proof}

\begin{lemma}\label{meager}
We have $\lim_{\lambda\to 0} p_B(Z_{\lambda})=0$. 
\end{lemma}
\begin{proof}
For each $\lambda$-gallery there is some $\mu(\lambda)$ such that 
its points are not contained in any $\mu$-gallery. 
Assume the contrary. Then there exists a sequence of 
$\lambda_j\to 0$ of nested $\lambda_j$-galleries. Their 
intersection point is an interior point of these arcs and thus 
it yields a point where the curve $\partial F$ has a self-intersection, 
which is a contradiction. Thus the claim follows. 

\vspace{0.2cm}
\noindent 
Since the number of $\lambda$-galleries is finite 
there is a sequence $\lambda_j\to 0$ such that $\lambda_j$-galleries 
are pairwise disjoint.   Thus 
\begin{equation}  \sum_{j}p_B(Z_{\lambda_j})\leq p_B(\partial F)\end{equation}
and hence  $\lim_{j}p_B(Z_{\lambda_j})=0$. 

\vspace{0.2cm}
\noindent 
Further any $Z_{\lambda}$ is contained into some  the union 
of $Z_{\lambda_j}$ for some $j>j(\lambda)$. The result follows.   
\end{proof}

\begin{lemma}\label{limsup}
We have $\lim_{\lambda\to 0}2\lambda N_{\lambda}(F,B)\leq p_B(\partial F)$. 
\end{lemma}
\begin{proof}
Let $P(\lambda)$ be the polygon associated to a 
maximal $\lambda$-necklace on $\partial F$. Then 
\begin{equation}  2\lambda N_{\lambda}(F,B) \leq p_B(P(\lambda))\end{equation}
For all $\lambda$ the set $\partial^+F_{\lambda}$  
is the union of finitely many Lipschitz curves and 
$P(\lambda)$ is a polygon inscribed in 
$\partial^+F_{\lambda}$. However, it might happen that 
$\partial^+F_{\lambda}$ has several components, possibly infinitely many. 

\vspace{0.2cm}
\noindent 
Recall that we defined in the proof of Proposition \ref{leng} the 
intermediary curve $W_{\lambda}=W_{\lambda}(Q)$ 
which is associated to a polygon  $Q$.  
We can define $W_{\lambda}(F)$ as the  Hausdorff limit of $W_{\lambda}(Q_n)$
where $Q_n$ is approximating $\partial F$.
Or else we can choose $Q$ which approximates closed 
enough to $\partial F$ so that the vertices of 
$P(\lambda)$ belong to $W_{\lambda}(Q)$. 

\vspace{0.2cm}
\noindent 
Moreover, $W_{\lambda}(Q)$ is now a closed curve, which might 
have self-intersections. The polygon $P(\lambda)$ is inscribed in 
$W_{\lambda}(Q)$ and we can associate disjoint arcs to different edges, 
since edges are associated to consecutive beads.  
Therefore  we have:  
\begin{equation}  p_B(P(\lambda))\leq p_B(W_{\lambda}(Q))\end{equation}
Then  the proof of Proposition \ref{leng} actually shows that   
\begin{equation}  p_B(W_{\lambda}(Q))\leq p_B(Q)+\lambda p_B(\partial B)\leq 
p_B(\partial F)+\lambda p_B(\partial B)\end{equation}
The inequalities above imply that 
\begin{equation} 2\lambda N_{\lambda}(\partial F)\leq p_B(\partial F)+\lambda p_B(\partial B)\end{equation}
and taking the limit when $\lambda$ goes to zero yields the claim. 
\end{proof}

\vspace{0.2cm}
\noindent 
{\em End of the proof of Theorem \ref{lip}}. 
By Lemma \ref{liminf} the limit is at least 
$p_B(\partial F\setminus Z_{\lambda})$, for any $\lambda$. 
Using Lemma \ref{meager} this lower bounds converges to $p_B(\partial F)$ 
when  $\lambda$ goes to zero. Then Lemma \ref{limsup} concludes 
the proof.

\section{Second order estimates}
The aim of this section is to understand better the rate of convergence 
in Theorem \ref{lim}. First, we have the very general upper bound below: 

\begin{proposition}
For any planar  simply connected domain $F$ 
with Lipschitz boundary we have 
\begin{equation}  2 \lambda N_{\lambda}(F,B) \leq p_B(\partial F)+\lambda p_B(\partial B)\end{equation}
\end{proposition}
\begin{proof} This is an immediate consequence of the proof of 
Lemma \ref{limsup}. 
\end{proof}

\vspace{0.2cm}
\noindent 
When $F$ is convex, we can obtain more effective estimates of the 
rate of convergence for the lower bound: 

\begin{proposition}\label{ineq}
For any symmetric oval $B$ and convex disk $F$ in the plane,  
the following inequalities hold true: 
\begin{equation}
p_B(F)-2\lambda  \leq 2 \lambda N_{\lambda}(F,B) \leq 
p_B(F)+\lambda p_B(\partial B)
\end{equation}
 \end{proposition}
\begin{proof}
By approximating the convex curve $\partial F$ 
by convex polygons we deduce the following extension of the 
classical tube formula to Minkowski spaces: 
\begin{equation}  p_B(\partial^+F_{\lambda})=p_B(F) + \lambda p_B(\partial B) \end{equation}
Notice that for non-convex $F$ we have only an inequality above. 

\vspace{0.2cm}
\noindent 
Let $B_1,\ldots, B_N$ be  a maximal necklace with beads which are translates 
of  $\lambda B$ and  $o_1,o_2,\ldots, o_N$ be their respective centers, 
considered in a cyclic order around $F$. 
Since $B_i\cap F$ contains at least one boundary point it follows that 
$o_i\in F_{\lambda}$ and $B_i\subset F_{2\lambda}$. 

\vspace{0.2cm}
\noindent 
Since $B$ and $F$ are convex it follows that $F_{\lambda}$ is convex. 
Therefore the polygon $P=o_1o_2\cdots o_N$ is convex since its vertices 
belong to $\partial^+F_{\lambda}$ and, moreover, $P\subset F_{\lambda}$.

\vspace{0.2cm}
\noindent 
It is not true in general that $P$ contains $F$, and we have 
to modify $P$.

\vspace{0.2cm}
\noindent 
If the necklace is incomplete,  
we can fill in the space left  by 
adjoining an additional translate $B^*_{N+1}$ homothetic to $B$ in 
the ratio $\lambda\mu$, with $\mu <1$, which has a common point with 
each one of $F, B_1$ and $B_N$. Set $o_{N+1}$ for its center.  

\vspace{0.2cm}
\noindent 
Now, we claim that the polygon $P^*=o_1o_2\cdots o_{N+1}$ contains $F$.  
In fact, $d_B(o_i,o_{i+1})\leq 2$ since $B_i$ and $B_{i+1}$ have  
a common point, which is at unit distance from the centers. But 
their interiors have empty intersection thus $d_B(o_i,o_{i+1})=2\lambda$ 
and the segment $|o_io_{i+1}|$  contains  one intersection point  
from  $\partial B_i \cap \partial B_{i+1}$.  
Furthermore, the same argument shows that 
$d_B(o_N,o_{N+1})=d_B(o_1,o_{N+1})=
(1+\mu)\lambda$ and each segment $|o_No_{N+1}|$ and  
$|o_{N+1}o_1|$ contains one boundary point from the corresponding 
boundaries intersections. Thus the boundary of $P^*$ is contained 
in  $\cup_{i=1}^{N+1}B_i\cup B^*_{N+1}$, the later being disjoint 
from the interior of $F$. This proves our claim.   
Remark that $P^*$ is not necessarily convex.

\vspace{0.2cm}
\noindent 
We know that $P\subset F_{\lambda}$ and 
$d_B(o_i,o_{i+1})\geq 2$ (for $i=1,2,\ldots, N$) 
because $B_i$ and $B_{i+1}$ have no common interior points. 
Since a convex curve surrounded by 
another curve is shorter than the containing one we obtain: 
\begin{equation} 
2\lambda N \leq p_B(P)\leq p_B(\partial^+F_{\lambda})=p_B(\partial F)+\lambda p_B(\partial B) 
\end{equation}

\vspace{0.2cm}
\noindent 
Next, recall that  $F\subset P^*$ and $d_B(o_i,o_{i+1})\leq 2$, where 
$i=1,2,\ldots, N+1$, since  
consecutive beads have at least one common point. 
This implies that:  

\begin{equation}  
p_B(\partial F)\leq p_B(P^*) < 2\lambda (N+1)
\end{equation} 
The two inequalities above prove the Proposition \ref{ineq}. 
\end{proof}

\vspace{0.2cm}
\noindent 
Consider a more general case when  $F$ is not necessary convex. 
We assume that $F$ is {\em regular}, namely that its boundary  
is the union of finitely many 
arcs with the property that each arc is either convex or concave.
The endpoints of these maximal arcs  are called vertices of $\partial F$. 
This is the case,
for instance, when $\partial F$ is a 
piecewise analytic curve. Moreover we will suppose that 
$F$ has positive reach. This is the case for instance when   
$F$ admits a support line 
through each vertex of $\partial F$, which leaves a neighborhood 
of the vertex in $F$ on one side 
of the half-plane.

\vspace{0.2cm}
\noindent 
The estimates for the rate of convergence  
will not be anymore sharp. By hypothesis, $\partial F$ 
can be decomposed into finitely many arcs $A_i$, $i=1,m$, 
which we call pieces, so that each piece is either convex or concave.

\begin{proposition}\label{nonc}
Consider a symmetric oval $B$ and a regular  topological disk 
$F$ of positive reach having $c(F)$ convex pieces and $d(F)$ 
concave pieces. Then the  following inequalities hold:  
\begin{equation}
p_B(F)- 2\lambda (2c(F)+p_B(\partial B)d(F)+ 3d(F))\leq 2 \lambda N_{\lambda}(F,B) \leq 
p_B(F)+ 2\lambda p_B(\partial B) 
\end{equation}
for $2\lambda< r(\partial F)$.  
 \end{proposition}
\begin{proof}
If ${\mathcal N}$ is a maximal necklace on $F$ denote by 
${\mathcal N}|_{A_j}$ its trace on the 
arc $A_j$, i.e.  one considers only those beads that touch $A_j$. 
Consider also maximal necklaces $M_{A_j}$
on each arc $A_j$, consisting of only those  beads sitting outside $F$
 which have common points to $A_j$.
Consider now the union of the maximal necklaces $M_{A_j}$. 
Beads of $M_{A_j}$ cannot intersect $\partial F$ since the reach is 
larger than $\lambda$. Moreover beads from different necklaces 
cannot intersect (according to Lemma \ref{consecutive}) unless 
the beads are consecutive beads i.e. one is the last  bead 
on $A_j$ and the other is the first bead on the next (in clockwise 
direction) arc $A_{j+1}$.  
 
\vspace{0.2cm}
\noindent
Therefore if we drop the last bead from each $M_{A_j}$ and take their union 
we obtain a necklace on $\partial F$. 
This implies that: 
\begin{equation}  \sum_{i=1}^m N_{\lambda}(A_i,B)-N_{\lambda}(F,B) \leq c(F)+d(F) \end{equation}

\vspace{0.2cm}
\noindent
We analyze convex arcs in the same manner as  we did for ovals in the 
previous Proposition. 
Since the arc $A_j$ has positive reach we can slide all beads to the left 
side. Add one more smaller bead in the right side which touches the arc 
at its endpoint, if possible. The centers of the beads form a polygonal line  
$P^*$ with  at most $N_{\lambda}(A_j,B)+1$ beads. Join its endpoints to the 
endpoints of the arc $A_j$ by two segments of length no larger than $\lambda$. 
This polygonal line surrounds the convex arc $A_j$ and thus its length is 
greater than $p_B(A_j)$. Therefore, for each convex arc $A_j$ we 
have:   
\begin{equation}  2\lambda (N_{\lambda}(A_j,B)+1) \geq p_B(A_j)\end{equation}

\vspace{0.2cm}
\noindent The next step is to derive similar estimates from below 
for  a concave arc $A_s$. Since the arc has positive reach we can slide 
all beads to its left side. If there is more space left to the right 
let us continue the arc $A_s$ by adding a short arc on its right side
along a  limit support line at the right endpoint so that we can add one more 
bead to our necklace which touches the completed arc $A_s^*$ at its endpoint.
We can choose this line so that the reach of $A_s^*$ is the same as that of $A_s$.

\vspace{0.2cm}
\noindent
Let the beads have centers $o_i$, 
$i=1,N+1$, where $N=N_{\lambda}(A_s,B)$, the last one being the center of the 
additional bead. 
Then $d_B(o_i,o_{i+1})=2\lambda$ and $d_B(o_i,A)=\lambda$, 
as in the convex case. The point is that the function $d_B(x, A)$ is 
not anymore convex, as it was for convex arcs. 
However, for any point $x\in |o_io_{i+1}|$ we have 
$d_B(x,A)\leq \min(d_B(x,o_i)+d_B(o_i,A), d_B(x,o_{i+1})+d(o_{i+1},A))
\leq 2\lambda$. 

\vspace{0.2cm}
\noindent
If $P^*$ denotes the polygonal line   $o_1o_2\cdots o_{N+1}$  
then $P^*\subset {A_s^*}_{2\lambda}$. 
Moreover the points which are opposite to 
the contacts between the beads and $A_s^*$ belong to ${A_{s}^*}_{2\lambda}$. 
Join in pairs the endpoints of $P^*$ and 
those of  ${A_s^*}_{2\lambda}$ by two segments of length $\lambda$  
and denote their union with $P^*$ by $\overline{P^*}$. 
The arc $A_s$ was considered concave of positive reach and this means 
that for small enough $\lambda< r(\partial F)/2$ the boundary 
$\partial {A_s^*}_{2\lambda}$ is still concave of positive reach. 
Looking from the opposite side ${A_s^*}_{2\lambda}$ is a 
convex arc. Moreover $\overline{P^*}$ encloses (from the opposite side) 
this convex arc and thus 
$p_B(\overline{P^*})\geq p_B({A_s^*}_{2\lambda})\geq p_B({A_s}_{2\lambda})$. 

\vspace{0.2cm}
\noindent
The formula giving the perimeter for the parallel has a version for the 
inward deformation of convex arcs, or equivalently, 
for  outward deformations of concave arcs, 
which reads as follows: 
\begin{equation}  p_B({A_s}_{2\lambda})=p_B(A_s)-2\lambda  p_B(X_{A_s})\end{equation}
where $X_{A_s}\subset \partial B$ is the image of $A_s$ 
by the Gauss map associated 
to $B$. As $X_{A_s}\subset \partial B$ we obtain   
\begin{equation}   p_B(A_s)-2\lambda p_B(\partial B)\leq p_B(P^*)+2\lambda\leq 
2\lambda N_{\lambda}(A_s,B)+4\lambda  \end{equation}

\vspace{0.2cm}
\noindent
Summing up these inequalities we derive the inequality from the statement. 
\end{proof}

\begin{remark}
The proofs above work for arbitrary convex $B$, not necessarily 
centrally symmetric. In this case, we could obtain: 
\begin{equation}p_{\overline{B}}(F) = 2\lim_{\lambda \to 0}\lambda N_{\lambda}(F,B) 
\end{equation}
where $\overline{B}=\frac{1}{2}(B-B)=\{z\in \R^2; {\rm \, there \,\,exist \,} 
x, y \in B, {\, \rm such \, that\,}  2z=x-y\}$. 
\end{remark}

\vspace{0.2cm}
\noindent 
Consider the set $N(F,B)$ of all positive integers that 
appear as $N_{\lambda}(F,B)$ for some $\lambda\in (0,1]$.

\begin{corollary}
If $F$ is regular and its boundary  has positive reach 
then large enough  consecutive 
terms in $N(F,B)$ are at most distance $11d(F)+2c(F)+4$ apart. 
When  $F$ is convex  consecutive terms in $N(F,B)$ are at most distance 
4 apart, if $B$ is not a parallelogram and 5 otherwise. 
\end{corollary}
\begin{proof}
Let us consider $F$ convex. Theorem 2 shows that 
\begin{equation}  \frac{p_{B}(F)}{2\lambda}-1 < N_{\lambda}(F,B) \leq 
 \frac{p_{B}(F)}{2\lambda} +  \frac{p_{B}(\partial B)}{2}\end{equation}
Moreover one knows that $p_B(\partial B)\leq 8$ (see \cite{BF} and references there) 
with equality only when $B$ is a parallelogram.
In particular any interval $(\alpha, \alpha+ \frac{p_B(\partial B)}{2}]$ 
contains at least one element of $N(F,B)$. If $c < d$ are two consecutive 
elements of $N(F,B)$ this implies that $c\in (d-\frac{p_B(\partial B)}{2}-1, d)$, 
and thus 
\begin{equation} d-c  <\frac{p_B(\partial B)}{2}+1 \leq 5\end{equation}
Since $c,d$ are integers it follows that $d-c\leq 4$, if $B$ is not a parallelogram.

\vspace{0.2cm}\noindent 
When $F$ is arbitrary the inequality in theorem 3 shows that any 
interval of length $11d(F)+2c(F)+4$ contains 
some $N_{\lambda}(F,B)$. We conclude as above.   
\end{proof}

\begin{corollary}\label{consecut} 
Consecutive terms in $N(B,B)\subset \Z_+$ are at most distance 
4 apart. 
\end{corollary}
\begin{proof}
If $F=B$ is a parallelogram then 
$N_{\lambda}(F,B)= 4\left[ \frac{1}{\lambda}\right]+4$ and thus 
$N(F,B)=4(\Z_+-\{0,1\})$.
\end{proof}

\begin{remark}
If $F$ is not convex then we can have gaps of larger size in $N(F,B)$. 
Take for instance $F$ having the shape of a staircase with $k$ stairs and 
$B$ a square. As in the remark above we can compute 
$N_{\lambda}(F,B)= 4k\left[ \frac{1}{\lambda}\right]+4$ and thus 
there are gaps of  size $4k$. 
\end{remark}

\section{Higher dimensions}

The previous results have generalizations to higher dimensions in terms of 
some Busemann-type areas defined by $B$. 
Theorem \ref{lim}, when $F=B$,   was extended in \cite{BLSZ} and 
(\cite{B}, 9.10).  The result involves the 
presence of an additional density factor which seems more complicated for $d >2$.

\vspace{0.2cm}
\noindent 
For a convex body $K$ in $\R^{d}$ one defines the translative 
packing density $\delta(K)$ to be the supremum of 
the densities of periodic packings by translates of $K$ 
and set $\Delta(K)=\frac{1}{\delta(K)}\vol(K)$.  Alternatively, 
 $\Delta(K)=\inf_{T,n}\vol(T)/n$ over all tori $T$ and 
integers $n$ such that there exists a packing with $n$ translates 
of $K$ in $T$, where $T$ is identified with a quotient of $\R^d$ 
by a lattice.   

\vspace{0.2cm}
\noindent 
We consider from now on that $B$ and $F$ are convex and smooth. 

\begin{proposition}
We have for a convex smooth  $F\subset \R^d$ and a centrally 
symmetric smooth domain $B\subset \R^d$ that  
\begin{equation}  \lim_{\lambda\to \infty} \lambda^{d-1} N_{\lambda}(F,B)=
\int_{\partial F}\frac{1}{\Delta(B\cap T_x)} d x\end{equation}
where $x\in \partial F$ and $T_x$ is  a hyperplane through the center of $B$ 
which is parallel to the tangent space at $\partial F$ in $x$.  
Here $B\cap T_x\subset T_x$ is identified to  
a $(d-1)$-dimensional domain in $\R^{d-1}$.   
\end{proposition}
\begin{proof}
The proof from (\cite{B} 9.10) can be adapted to work in 
this more general situation as well.  
\end{proof}

Although the present methods do not extend to general arbitrary domains with 
rectifiable boundary the  previous proposition seem to generalize at least when 
the boundary has positive reach.

\begin{remark} We have an obvious upper bound 
\begin{equation}   N_{\lambda}(F,B) \leq \lambda^{-d}  \frac{\vol(F_{2\lambda})-\vol(F)}
{\vol(B)} =\lambda^{1-d}{\rm area}_B(\partial F) + {o}(\lambda^{1-d})\end{equation}
which follows from the inclusion  $\cup_{i=1}^NB_{i} \subset F_{2\lambda}$
with  $B_i$ having disjoint interiors and the Steiner formula (see \cite{G}). 
\end{remark}

\section{Remarks and  conjectures}
The structure of the sets $N(F,B)$ is largely unknown. 
One can prove that when $F$ is convex and both $F$ and $B$ are smooth 
then $N(F,B)$ contains all integers from $N_1(F,B)$ on, at least 
in dimension 2. For general $F$ we saw that we could have gaps. 
It would be interesting to know whether $N(B,B)$ contains 
all sufficiently large integers when $F=B$ is a convex  
domain and not a parallelohedron.  It seems that Corollary \ref{consecut} 
can be generalized to higher dimensions as follows: 

\begin{conjecture}
The largest distance between consecutive elements of $N(F,B)$, where 
$F$ is convex is at most $2^d$ with equality when $F=B$ is a parallelohedron. 
\end{conjecture}

\vspace{0.2cm}
\noindent 
Another natural problem is to understand the higher order terms in the 
asymptotic estimates. Or, it appears that second order terms from section 4  
are actually oscillating according to the inequalities in proposition 
\ref{ineq} as below:

\begin{corollary}
For convex $F$ we have 
\begin{equation} -2 \leq l_-(F,B)=\liminf_{\lambda\to 0}\frac{2\lambda N_{\lambda}(F,B) - p_B(F)}{\lambda}\leq \limsup_{\lambda\to 0}\frac{2\lambda N_{\lambda}(F,B) - p_B(F)}{\lambda}=l_+(F,B)\leq \frac{p_B(\partial B)}{2}\end{equation}
\end{corollary}

\vspace{0.2cm}
\noindent 
The exact meaning of $l_-(F,B)$ and $l_+(F,B)$ is not clear.
Assume that $F=B$. We computed: 
\begin{enumerate}
\item  
If $B$ is a disk then $l_-(F,B)=-2$ and $l_+(F,B)=0$;
\item 
If $B$ is a square then $l_-(F,B)=0$ and $l_+(F,B)=4$; 
\item If $B$ is a regular hexagon then $l_-(F,B)=0$ and $l_+(F,B)=3$; 
\item If $B$ is a triangle then $l_-(F,B)=0$ and $l_+(F,B)=3$. 
\end{enumerate}

\vspace{0.2cm}
\noindent 
There are various other invariants related to the second order terms. 
Set
\begin{equation} J_k=\{ \lambda \in (0,1]; {\rm \,\,    there \,\, exists \,\, a \,\, complete \,\, \lambda-necklace\,}\,\, B_1,\ldots, B_{k}\}, \,\,\, 
I_k=\{\lambda \in (0,1]; N_{\Lambda}(F,B)=k\}\end{equation} 
so that $J_k\subset I_k$.  Then 
$I_k$ are disjoint connected intervals  but we don't know whether 
this is equally true for $J_k$. It seems that $J_k$ are singletons when 
$B$ is a round disk.  


\vspace{0.2cm}
\noindent 
Let $\{r\}=r-[r]$ denote the fractionary part of $r$. 

\begin{conjecture} 
There exists  some constant $c=c(B)$ such that the 
following limit exists  
\begin{equation}   \lim_{r\to \infty, \{r\}=\alpha}2N_{c(B)/r}(F,B) - r= 
\varphi(\alpha) \end{equation}
where $\varphi:[0,1)\to [-2, p_B(\partial B)]$ is a 
right continuous function  with finitely many singularities. 
If $F$ and $B$ are polygons then 
$\varphi$ is linear on each one of its intervals of continuity.   
\end{conjecture}

\vspace{0.4cm}
\noindent 
{\bf Acknowledgments.} The authors are indebted to the referee for 
suggesting greater generality for the main result and 
to Herv\'e Pajot for several discussions concerning the 
results of this paper. The second author acknowledges support from  the grant 
ANR Repsurf: ANR-06-BLAN-0311.

\begin{small}
\bibliographystyle{plain}

\begin{thebibliography}{4}

\bibitem{Ban}
V.Bangert, {\em Sets with positive reach}, Archiv Math. 38(1982), 54-57. 

\bibitem{BF}
V.Boju and L.Funar, {\em Generalized Hadwiger numbers for symmetric ovals}, 
 Proc. Amer. Math. Soc.  119(1993),  931--934. 



\bibitem{B}
K.B\"or\"oczky Jr, {Finite packing and covering}, Cambridge Tracts in Mathematics, 154, Cambridge University Press, Cambridge, 2004. 


\bibitem{BLSZ}
K.B\"or\"oczky Jr., D.G.Larman,  S.Sezgin and C.Zong,
{\em On generalized kissing numbers and blocking numbers}, 
III International Conference in "Stochastic Geometry, Convex Bodies and Empirical Measures", Part II (Mazara del Vallo, 1999).
Rend. Circ. Mat. Palermo (2) Suppl. No. 65, part II (2000), 39--57. 


\bibitem{Ce}
L.Cesari, Surface area, Ann. Math. Studies 35, Princeton Univ. Press, 
Princeton, 1956.  
 
\bibitem{E}
P.Erd\"os, {\em Some remarks on the measurability of certain sets}, 
Bull. Amer. Math.Soc. 51(1945), 728-731. 

\bibitem{Fed}
H.Federer, {\em Curvature mesaures}, Trans.Amer. Math. Soc. 
93(1959), 418-491. 

\bibitem{Fe1}
L. Fejes Toth, {\em  
\"Uber eine affininvariante Masszahl bei Eipolyedern}, 
Studia Sci. Math. Hungar. 5(1970), 173--180. 

\bibitem{Fe2}
L. Fejes Toth, {\em On Hadwiger numbers and Newton numbers of a convex body}, 
  Studia Sci. Math. Hungar.  10(1975), 111--115. 

\bibitem{Fe}
S.Ferry, {\em When $\varepsilon$-boundaries are manifolds}, 
Fundamenta. Math. 90(1976), 199-210. 


\bibitem{Fu}
J.H.G.Fu, 
{\em Tubular neighborhoods in Euclidean spaces}, Duke Math. J. 
52(1985), 1025-1046. 

\bibitem{GP}
R.Gariepy and W.D.Pepe, {\em On the level sets of a distance 
function in a Minkowski space}, Proc. Amer. Math. Soc. 31(1972), 255-259. 


\bibitem{Go1}
S.G\`olab, {\em Sur la longueur de l'indicatrice dans la g\'eom\'etrie 
plane de Minkowski}, Colloq. Math. 15(1956), 141-144. 

\bibitem{Go2}
S.G\`olab, {\em Some metric problems in the geometry of Minkowski}, 
(Polish, French summary), Prace Akad. G\'orniczej w Krakowie 6(1932), 1-79. 
 
 


\bibitem{G}
A.Gray, {Tubes}, Progress in Mathematics, 221, 
Birkh\"auser Verlag, Basel, 2004.  


\bibitem{Ly}
A.Lytchak, {\em Almost convex curves}, Geom.Dedicata 115(2005), 201-218. 


\bibitem{MSW1}
H.Martini, K.J.Swanepoel and G. Weiss,   
{\em The geometry of Minkowski spaces---a survey} I,  Expositiones Math.  
19(2001),  97--142.




\bibitem{OP}
I.Ya.Oleksiv and N.I.Pesin, {\em Finiteness of Hausdorff measure of level sets 
of bounded subsets of Euclidean space}, Mat.Zametki 37(1985), no. 3, 422-431. 

\bibitem{St}
L.L.Stach\'o, {\em On the volume function of parallel sets}, 
Acta. Sci. Math. 38(1976), 365-374. 

\bibitem{St2}
L.L.Stach\'o, {\em On curvature measures}, 
Acta. Sci. Math. 41(1979), 191-207. 

\bibitem{T}
A.C.Thompson, Minkowski Geometry, Encyclopedia of Mathematics and its Applications, 63,  Cambridge University Press, Cambridge, 1996.

\bibitem{Tricot}
C.Tricot Jr., {\em Two definitions of fractional dimensions}, 
Math.Proc.Cambridge Philos.Soc. 91(1982), 57-74. 


\bibitem{Val}
F.A. Valentine, Convex sets, McGraw-Hill Series in Higher Math., 1964. 
\end{thebibliography}

\end{small}

\end{document}